\documentclass{article}
\usepackage{amsmath, url}
\usepackage{amsfonts}
\usepackage{amsthm}
\usepackage{latexsym}
\usepackage{amssymb}
\usepackage{verbatim}
\usepackage{enumerate}
\font\logic=msam10 at 10pt

\newcommand{\restrict}{\mbox{\logic\char'026}}

\newcommand{\less}{\mathord{<}}

\newcommand{\ms}{\mathrm{ms}}

\newtheorem{thrm}{Theorem}[section]
\newtheorem{lem}[thrm]{Lemma}
\newtheorem{cor}[thrm]{Corollary}
\newtheorem{prop}[thrm]{Proposition}
\newtheoremstyle{hdefinition}%
  {\topsep}%
  {\topsep}%
  {\upshape}%    Definitions, etc not in italics.
  {}%
  {\bfseries}%
  {.}%           would be better with period.
  { }%
  {\thmnumber{#2 }\thmname{#1}\thmnote{ \rm(#3)}}%
\newtheoremstyle{hclaim}%
  {\topsep}%
  {\topsep}%
  {\itshape}%
  {}%
  {\bfseries}%
  {.}%          better with period.
  { }%
  {\thmname{#1}\thmnote{ \rm#3}}%
\theoremstyle{hclaim}
\newtheorem*{claim*}{Claim}
\theoremstyle{hdefinition}
\newtheorem{df}[thrm]{Definition}

\newtheorem{example}[thrm]{Example}
\newtheorem{conj}[thrm]{Conjecture}
\newtheorem{remark}[thrm]{Remark}
\newtheorem{ques}[thrm]{Question}

\newtheorem{prob}[thrm]{Problem}
\theoremstyle{hclaim}

\newcommand{\Ord}{\mathrm{Ord}}

\newcommand{\Col}{\mathrm{Col}}

\newcommand{\cF}{\mathcal{I}}
\newcommand{\cI}{\mathcal{I}}

\newcommand{\cL}{\mathcal{L}}

\newcommand{\range}{\mathrm{range}}

\begin{document}

\title{Scott processes}
\author{Paul B. Larson\thanks{Research supported in part by NSF Grants
  DMS-0801009 and DMS-1201494. The author thanks John Baldwin for his many comments on earlier drafts, and Leo Harrington for discussing the
  material in this paper with him.}}

\maketitle

\begin{abstract}
The Scott process of a relational structure $M$ is the sequence of sets of formulas given by the Scott analysis of $M$.
We present axioms for the class of Scott processes of structures in a relational vocabulary $\tau$, and use them to give a proof of an unpublished theorem of Leo Harrington from the 1970's, showing that a counterexample to Vaught's Conjecture
has models of cofinally many Scott ranks below $\omega_{2}$. Our approach also gives a theorem of Harnik and Makkai, showing that if there exists a counterexample to Vaught's Conjecture, then there is a counterexample whose uncountable models have the same $\cL_{\omega_{1}, \omega}(\tau)$-theory, and which has a model of Scott rank $\omega_{1}$.
%Finally, we prove a theorem of Sacks showing that a counterexample to Vaught's Conjecture has distinct models of the same Scott rank.
Moreover, we show that if $\phi$ is a sentence of $\cL_{\omega_{1}, \omega}(\tau)$ giving
rise to a counterexample to Vaught's Conjecture, then for every limit ordinal $\alpha$ greater than the quantifier depth of $\phi$ and below $\omega_{2}$, $\phi$ has a model of Scott rank $\alpha$.
% and infinitely many models whose Scott ranks lie in the interval $[\alpha, \alpha + \omega)$,
%and that for club many ordinals $\alpha$ below each of $\omega_{1}$ and $\omega_{2}$, $\phi$ has at least two nonisomorphic models of Scott rank $\alpha$.
\end{abstract}

\section{Introduction}

We fix for this paper a
%countable\footnote{The assumption that $\tau$ is countable is only important for keeping track of the cardinalities of the $\Phi_{\alpha}$'s, especially when $\alpha$ is countable. It should probably be removed here and inserted where necessary.}
relational vocabulary $\tau$, and distinct variable symbols $\{ x_{n} : n < \omega\}$. For notational convenience, we assume that $\tau$ contains a $0$-ary relation symbol, as well as the binary symbol $=$, which is always interpreted as equality.
We refer the reader to \cite{Keisler, Hodges, Marker} for the definition of the language $\cL_{\infty, \omega}(\tau)$ and the languages $\cL_{\kappa, \omega}(\tau)$, for $\kappa$ an infinite cardinal. In this paper, all formulas will have only finitely many free variables. Formally, we consider conjunctions and disjunctions of formulas as unordered, even when we write them as indexed by an ordered set (in this way, for instance, a formula in $\cL_{\omega_{2}, \omega}(\tau)$ becomes a member of $\cL_{\omega_{1}, \omega}(\tau)$ in a forcing extension in which the $\omega_{1}$ of the ground model is countable).
%\footnote{This is important in Section \ref{fabsection}.}
 %$\top$, which is to be interpreted as true in all structures that we care about.
 We begin by recalling the standard definition of the Scott process corresponding to a $\tau$-structure $M$ (see \cite{Hodges, Marker}), slightly modified to require the sequences $\bar{a}$ to consist of distinct elements.

\begin{df}\label{Scottformdef} Given a $\tau$-structure $M$ over a relational vocabulary $\tau$, we define for each finite ordered tuple $\bar{a} = \langle a_{0},\ldots,a_{|\bar{a}|-1}\rangle$ of distinct elements of $M$ and each ordinal $\alpha$
the $|\bar{a}|$-ary formula $\phi^{M}_{\bar{a}, \alpha} \in \cL_{|M|^{+}, \omega}(\tau)$, as follows.
\begin{enumerate}
\item\label{Sdefone} Each formula $\phi^{M}_{\bar{a}, 0}$ is the conjunction of all expressions of the two following forms:
\begin{itemize}
\item $R(x_{f(0)},\ldots,x_{f(k-1)})$, for $R$ a $k$-ary relation symbol from $\tau$ and $f$ a function from $k$ to $|a|$, such that $M \models R(a_{f(0)},\ldots,a_{f(k-1)})$,
\item $\neg R(x_{f(0)},\ldots,x_{f(k-1)})$, for $R$ a $k$-ary relation symbol from $\tau$ and $f$ a function from $k$ to $|a|$, such that $M \models \neg R(a_{f(0)},\ldots,a_{f(k-1)})$.
%\item $\neg R(x_{i_{1}},\ldots,x_{i_{k}})$, for $R$ a $k$-ary relation symbol from $\tau$, $\{i_{1},\ldots,i_{k}\}$ a subset of $|a|$, and $M \models %\neg R(a_{i_{1}},\ldots,a_{i_{k}})$;
\end{itemize}
%atomic and negated atomic formulas satisfied by $\bar{a}$ in $M$, using the free variables $x_{0},\ldots,x_{|\bar{a}|-1}$;
\item Each formula $\phi^{M}_{\bar{a}, \alpha + 1}$ is the conjunction of the following three formulas:
\begin{itemize}
\item $\phi^{M}_{\bar{a}, \alpha}$,
\item $\bigwedge_{c \in M \setminus \{a_{0},\ldots,a_{|\bar{a}|-1}\}} \exists x_{|a|} \phi^{M}_{\bar{a}^{\frown}\langle c \rangle,\alpha}$,
\item $\forall x_{|\bar{a}|} \not\in \{x_{0},\ldots,x_{|a|-1}\} \bigvee_{c \in M \setminus \{a_{0},\ldots,a_{|\bar{a}|-1}\}} \phi^{M}_{\bar{a}{\frown}\langle c\rangle,\alpha}$.
\end{itemize}
\item For limit ordinals $\beta$, $\phi^{M}_{\bar{a},\beta} = \bigwedge_{\alpha < \beta}\phi^{M}_{\bar{a},\alpha}$.
\end{enumerate}
We call $\phi^{M}_{\bar{a}, \alpha}$ the \emph{Scott formula} of $\bar{a}$ \emph{in} $M$ \emph{at level} $\alpha$.
\end{df}

%\begin{remark} In item (\ref{Sdefone}) of Definition \ref{Scottformdef}, we mean the formulas in the variables $x_{0},\ldots,x_{|\bar{a}|-1}$ which hold in $M$ %when the members of $\bar{a}$ are substituted for $x_{0},\ldots,x_{|\bar{a}|-1}$ in their given order.
%\end{remark}

The following well-known fact can be proved by induction on $\alpha$ (see Theorem 3.5.2 of \cite{Hodges}). Again, we refer the reader to \cite{Keisler, Hodges, Marker} for the definition of the \emph{quantifier depth} of a formula, and note that each formula $\phi^{M}_{\bar{a},\alpha}$ as defined above has quantifier depth exactly $\alpha$.

\begin{thrm}\label{wkh} Given $\tau$-structures $M$ and $N$, $n \in \omega$, an ordinal $\alpha$ and $n$-tuples $\bar{a}$ from $M$ and $\bar{b}$ from $N$, each consisting of distinct elements, $\phi^{M}_{\bar{a},\alpha} = \phi^{N}_{\bar{b}, \alpha}$ if and only if, for each $n$-ary $\cL_{\infty, \omega}(\tau)$ formula $\psi$ of
quantifier depth at most $\alpha$, $\bar{a}$ satisfies $\psi$ in $M$ if and only if $\bar{b}$ satisfies $\psi$ in $N$.
\end{thrm}

\begin{df} Given a $\tau$-structure $M$ and an ordinal $\beta$, we let $\Phi_{\beta}(M)$ denote the set of all formulas of the form $\phi^{M}_{\bar{a},\beta}$, for $\bar{a}$ a finite tuple of distinct elements of $M$. We call the class-length sequence $\langle \Phi_{\alpha}(M) : \alpha \in \Ord \rangle$ the \emph{Scott process of} $M$.
%The \emph{class-sized Scott process} of $M$ is the class of sets $\Phi_{\alpha}(M)$, for all ordinals $\alpha$.
\end{df}

This paper studies the class of Scott processes of all $\tau$-structures (or, more formally, the class of set-length initial segments of Scott processes of $\tau$-structures). Section \ref{fpsec} introduces an array of sets of formulas (properly) containing all the formulas appearing in the Scott process of any $\tau$-structure, and vertical and horizontal projection functions acting on this array. Section \ref{spsec} introduces our general notion of a Scott process (i.e., without regard to a fixed $\tau$-structure). Section \ref{ccsec} develops some of the basic consequences of this definition, and Section \ref{rssec} defines the rank of a Scott process. The material in these two sections checks that Scott process in general, as defined here, satisfy various basic properties of Scott processes of $\tau$-structures. Section \ref{threadsec} shows that a Scott process of countable length whose last level is countable is an initial segment of the Scott process of some $\tau$-structure. Section \ref{arbcardsec} shows how to build models of cardinality $\aleph_{1}$, for certain Scott processes (roughly, those corresponding to Scott sentences). Section \ref{qsec} develops more basic material on Scott processes, studying the way they reflect finite blocks of existential quantifiers. Section \ref{ellsec} looks at extending Scott processes of limit length. Section \ref{fabsection} is largely disjoint from the rest of the paper,
and presents an argument showing that in some cases (for instance, counterexamples to Vaught's Conjecture in $\cL_{\omega_{1},\omega}(\tau)$)
%and moreover, analytic classes of structures having uncountable countable models but not perfectly many)
a Scott process which exists in a forcing extension can be shown to exist in the ground model. Put together, the material in Sections \ref{arbcardsec}, \ref{ellsec} and \ref{fabsection} gives Harrington's theorem that a counterexample to Vaught's Conjecture has models of cofinally many Scott ranks below $\omega_{2}$. Our proof gives slightly more than Harrington's theorem as commonly stated, although we do not know if our version of his result is new.

The main results of the paper are summarized in Theorem \ref{Harringtonplus}. A proof of the Harnik-Makkai theorem on minimal counterexamples to Vaught's Conjecture is outlined in Remark \ref{hmremark}.

The material in this paper was inspired by the slides of a talk given by David Marker on Harrington's theorem \cite{MarkerHar}. Our proof is different in some respects from the proof outlined there. Marker's talk outlines a recursion-theoretic argument, assuming the existence of a counterexample $\phi$ to Vaught's Conjecture, for finding a sentence in $\cL_{\omega_{2}, \omega}$ which will be the Scott sentence of a model of $\phi$ (of suitably high Scott rank) in a forcing extension collapsing $\omega_{1}$. This part of the proof is replaced here by a forcing-absoluteness argument in Section \ref{fabsection} (essentially equivalent versions of these arguments appear in Section 1 of \cite{Hjorth96}). The remainder of Harrington's proof builds a model of this Scott sentence. This we do in Section \ref{arbcardsec}, guided by the argument in Marker's slides.

Another, different, proof of Harrington's theorem is due to appear in \cite{BaldwinFriedmanKoerwienLaskowski}.

%Section \ref{issuesec} lists some questions and other unresolved issues.

%The Scott rank of $M$ is
%the least $\alpha$ such that the formulas at level $\alpha + 1$ are equivalent for any pair of tuples whose formulas at level $\alpha$ are equivalent. The %sentence asserting this fact about the Scott formulas at levels $\alpha$ and $\alpha + 1$ is the Scott sentence of $M$; any two countable models are isomorphic %if and only if they have the same Scott sentence (formally, we think of the conjunctions and disjunctions in our language as unordered, so that there is no issue %of
%two distinct Scott sentences describing isomorphic models via a permutation of formulas).
%%; alternately we could fix wellorderings of minimal length for every set and require, in the ground model at least, that all Scott formulas respect these %ordering).

%\footnote{Not sure what to do with this sentence yet : We let $\mathcal{S}(\alpha)$ denote the set of all Scott formulas of level $\alpha$ for tuples in models of $\phi$.}

\section{Formulas and projections}\label{fpsec}

 For each $n \in
\omega$, let $X_{n}$ denote the set $\{ x_{m} : m < n\}$, and let $i_{n}$ denote the identity function on $X_{n}$. For all $m \leq n \in \omega$, we
let $\cF_{m,n}$ denote the set of injections from $X_{m}$ into $X_{n}$.

%We let $m(\tau)$ denote the least $m$ such that $\tau$ contains an $m$-ary relation, and assume that $m(\tau) > 0$.

%\begin{remark}
%\end{remark}

We start by defining a class of formulas which contains every formula appearing in the Scott process of any $\tau$-structure (see Remark \ref{mcont}). The sets $\Psi_{\alpha}$ defined below also contain formulas that do not appear in the Scott process of any $\tau$-structure. Definition \ref{maxdef} takes advantage of this extra generality, and in any case strengthening the definition to rule out such formulas would raise issues that we would rather defer. For the moment, the important point is that the sets $\Psi_{\beta}$ ($\beta \in \Ord$) are small enough to carry the projection functions $V_{\alpha, \beta}$ and $H^{n}_{\alpha}$ defined below.

%the domains for the functions $V_{\alpha, \beta}$ and $H_{\alpha}$ defined below.

\begin{df}\label{psidef} We define, for each ordinal $\alpha$ and each $n \in \omega$, the sets $\Psi_{\alpha}$ and $\Psi^{n}_{\alpha}$, by recursion on $\alpha$, as follows.
\begin{enumerate}
%\item\label{fone} Each $\Psi_{\alpha}$ is a set of $\cL_{\infty, \omega}(\tau)$ formulas of quantifier depth $\alpha$. We let $\Phi^{n}_{\alpha}$ denote the %members of $\Phi_{\alpha}$ whose free variables are exactly $X_{n}$, and stipulate that $\Phi_{\alpha} = \bigcup_{n \in \omega}\Phi^{n}_{\alpha}$, and that %each $\Phi^{n}_{\alpha}$ is closed under permutations of $X_{n}$.

\item For each $n \in \omega$, $\Psi^{n}_{0}$ is the set of all conjunctions consisting of, for each atomic $\tau$-formula using variables from $X_{n}$, either the formula or its negation, including an instance of the formula $x_{i} \neq x_{j}$ for each pair of distinct $x_{i}, x_{j}$ from $X_{n}$.

\item\label{nextstep} For each ordinal $\alpha$ and each $n \in \omega$, $\Psi^{n}_{\alpha+1}$ is the set of formulas $\phi$ for which there exist a formula
$\phi' \in \Psi^{n}_{\alpha}$ and a subset $E$ of $\Psi^{n+1}_{\alpha}$ such that
%\begin{enumerate}
%\item for some $n \in \omega$, the free variables of $\phi$ and $\phi'$ are both equal to $X_{n}$, and each element of $P(\phi)$ has
%$X_{n+1}$ as its set of free variables;
%\item
$\phi$ is the conjunction of $\phi'$ with the following two formulas.
\begin{enumerate}
 \item $\bigwedge_{\psi \in E} \exists x_{n} \psi$;
 \item $\forall x_{n} (x_{n} \not\in \{x_{0},\ldots,x_{n-1}\} \rightarrow \bigvee_{\psi \in E} \psi)$.
\end{enumerate}

\item\label{pathstage} For each limit ordinal $\alpha$ and each $n \in \omega$, $\Psi^{n}_{\alpha}$ is the set of
conjunctions which consist of exactly one formula $\psi_{\beta}$ from each $\Psi^{n}_{\beta}$, for $\beta < \alpha$, satisfying the following
conditions.
\begin{enumerate}
\item For each $\beta < \alpha$, $\psi_{\beta}$ is the formula $\phi'$ with respect to $\psi_{\beta + 1}$, as in condition (\ref{nextstep}).
\item For each limit ordinal $\beta < \alpha$, $\psi_{\beta} = \bigwedge\{ \psi_{\gamma} : \gamma < \beta\}$.
%formulas of the form $\bigwedge_{\beta < \alpha}\psi_{\beta}$, where each $\psi_{\beta}$ is an element of $\Phi^{n}_{\beta}$.
\end{enumerate}
\item For each ordinal $\alpha$, $\Psi_{\alpha} = \bigcup_{n \in \omega}\Psi^{n}_{\alpha}$.

\end{enumerate}
\end{df}

We can think of the sets $\Psi^{n}_{\alpha}$ as forming an array, with the rows indexed by $\alpha$ and the columns indexed by $n$. In the rest of this section we define the functions $V_{\alpha, \beta}$, which map between rows while preserving column rank, and the functions $H^{n}_{\alpha}$ which map between columns while preserving row rank.

\begin{remark}\label{exactfree}
  Each $\Psi_{\alpha}$ is a set of $\cL_{\infty, \omega}(\tau)$ formulas of quantifier depth $\alpha$, so the sets $\Psi_{\alpha}$ are disjoint for distinct $\alpha$. Similarly, for each $n \in \omega$ and each ordinal $\alpha$, $X_{n}$ is the set of free variables for each formula in each $\Psi^{n}_{\alpha}$.
\end{remark}

\begin{remark} As we require our vocabulary to contain a $0$-ary relation as well as the binary relation $=$, $\Psi^{n}_{\alpha}$ is nonempty for each ordinal $\alpha$ and each $n \in \omega$.
\end{remark}

%\begin{remark} The unique element of $\Psi^{0}_{0}$ is the sentence symbol $\top$, which was introduced at the beginning of the paper.
%The set $\Psi^{0}_{\alpha}$ is nonempty for each ordinal $\alpha$.
%\end{remark}

\begin{df}\label{pdef} For each ordinal $\alpha$, and each formula $\phi$ in $\Psi_{\alpha + 1}$, we let $E(\phi)$ denote the set $E$ from condition (\ref{nextstep}) of Definition \ref{psidef}.
\end{df}

\begin{remark}\label{mcont} If $M$ is a $\tau$-structure, $\alpha$ is an ordinal and $\bar{a}$ is a finite tuple of distinct elements of $M$, then the Scott formula of $\bar{a}$ in $M$ at level $\alpha$ defined in Definition \ref{Scottformdef} (i.e., $\phi^{M}_{\bar{a}, \alpha}$) is an element of $\Psi^{|\bar{a}|}_{\alpha}$. It follows that $\Phi_{\alpha}(M) \subseteq \Psi_{\alpha}$.
\end{remark}

The functions $V_{\alpha, \beta}$, as defined below, are the \emph{vertical projection functions}.

\begin{df}\label{vdef} The functions $V_{\alpha, \beta} \colon \Psi_{\beta} \to \Psi_{\alpha}$, for all pairs of ordinals $\alpha \leq \beta$ are defined as follows.
\begin{enumerate}
\item Each function $V_{\alpha, \alpha}$ is the identity function on $\Psi_{\alpha}$.

%\item\label{vone} Each $V_{\alpha, \beta}$ is a function from $\Phi_{\beta}$ to $\Phi_{\alpha}$ (the \emph{vertical projection}) where $\phi$ and $V_{\alpha, %\beta}(\phi)$ have the same free variables, and each $V_{\alpha, \alpha}$ is the identity function on $\Phi_{\alpha}$;

\item\label{vstep} For each ordinal $\alpha$, and each $\phi \in \Psi_{\alpha + 1}$, $V_{\alpha, \alpha+1}(\phi)$ is the
first conjunct of $\phi$, i.e., the formula $\phi'$ in condition (\ref{nextstep}) of Definition \ref{psidef}.

% i.e., $V_{\alpha, \alpha+1}(\phi', \phi)$ if and only if $\phi'$ is the first part of the conjunction forming $\phi$.

\item\label{pathcond} For each limit ordinal $\beta$, each formula $\phi \in \Psi_{\beta}$, and each $\alpha < \beta$, $V_{\alpha, \beta}(\phi)$ is the unique conjunct of $\phi$ in $\Psi_{\alpha}$.

     %has the form $\bigwedge_{\alpha < \gamma} \psi_{\alpha}$, where $V_{\alpha, \beta}(\psi_{\beta}) = \psi_{\alpha}$ for all $\alpha < \beta < \gamma$, and %$V_{\alpha, \gamma}(\phi) = \psi_{\alpha}$ for each $\alpha < \gamma$.

\item\label{vdir0} For all ordinals $\alpha < \beta$, $V_{\alpha, \beta + 1} = V_{\alpha, \beta} \circ V_{\beta, \beta + 1}$.

%conjunction of a $V_{\alpha, \beta}$-chain consisting of one member from each $\Phi_{\alpha}$, $\alpha < \gamma$, and for each $\alpha < %\gamma$ and each $\psi \in V_{\alpha}$, $V_{\alpha, \gamma}(\psi, \phi)$ if and only if $\psi$
%    is a member of the chain giving rise to $\phi$.
\end{enumerate}
\end{df}

\begin{remark}\label{vdir} Conditions (\ref{nextstep}) and (\ref{pathstage}) of Definition \ref{psidef} imply the following stronger version of condition (\ref{vdir0}) of
Definition \ref{vdef} : for all ordinals $\alpha \leq \beta \leq \gamma$, $V_{\alpha, \gamma} = V_{\alpha, \beta} \circ V_{\beta, \gamma}$.
\end{remark}

\begin{remark}\label{vone}
%Each $V_{\alpha, \beta}$ is a function from $\Psi_{\beta}$ to $\Psi_{\alpha}$, and
For all ordinals $\alpha \leq \beta$, each $n \in \omega$, and each $\phi \in \Psi^{n}_{\beta}$, $V_{\alpha, \beta}(\phi)$ is in $\Psi^{n}_{\alpha}$, so $\phi$ and $V_{\alpha, \beta}(\phi)$ have the same free variables.
\end{remark}

\begin{remark} Since the domains of the functions $V_{\alpha, \beta}$ are disjoint for distinct $\beta$, one could drop $\beta$ and simply write $V_{\alpha}$ (which would then be a definable class-sized function from $\bigcup_{\beta \in (\Ord \setminus \alpha)} \Psi_{\beta}$ to $\Psi_{\alpha}$). We retain both subscripts for clarity.
\end{remark}

We define the \emph{horizontal projection functions} as follows.

\begin{df}\label{hdef}
  The functions $H^{n}_{\alpha}$, for each ordinal $\alpha$ and each $n \in \omega$, are defined recursively on $\alpha$, as follows.
   \begin{enumerate}
   \item The domain of each $H^{n}_{\alpha}$ consists of all pairs $(\phi, j)$, where $\phi \in \Psi^{n}_{\alpha}$ and, for some $m \leq n$, $j\in \cF_{m,n}$.
       % and each such value $H_{\alpha}(\phi, j)$ is an element of $\Psi^{m}_{\alpha}$.

   \item For all $m \leq n$ in $\omega$, all formulas $\phi \in \Psi^{n}_{0}$, and all $j \in \cF_{m,n}$, $H^{n}_{0}(\phi, j)$ is the conjunction of all conjuncts from $\phi$ whose variables are contained in the range of $j$, with these variables replaced by their $j$-preimages.

   \item\label{projfact} For each ordinal $\alpha$, each $m \leq n$ in $\omega$, each $\phi \in \Psi^{n}_{\alpha + 1}$, and
   each $j \in \cF_{m,n}$, $H^{n}_{\alpha + 1}(\phi, j)$ is the formula $\psi \in \Psi^{m}_{\alpha+1}$ such that $$V_{\alpha, \alpha + 1}(\psi) = H^{n}_{\alpha}(V_{\alpha, \alpha+1}(\phi),j)$$ and $E(\psi) = H^{n+1}_{\alpha}[E(\phi) \times \{j \cup \{(x_{m},y)\} \mid y \in (X_{n+1} \setminus \range(j))\}]$.

   \item\label{hsix} For each limit ordinal $\alpha$, each $m \leq n$ in $\omega$, each $j \in \cF_{m,n}$ and each $\phi \in \Psi^{n}_{\alpha}$,
   $$H^{n}_{\alpha}(\phi, j) = \bigwedge \{ H^{n}_{\beta}(V_{\beta, \alpha}(\phi), j) : \beta < \alpha \}.$$
   \end{enumerate}
\end{df}

\begin{remark} Since the domains of the functions $H^{n}_{\alpha}$ are disjoint for distinct pairs $(\alpha, n)$, one could drop $\alpha$ and $n$ and simply write $H$. We retain them for clarity.
\end{remark}

\begin{remark}\label{hidrem}
   For all ordinals $\alpha$, all $m \leq n$ in $\omega$, all $j \in \cF_{m,n}$ and all $\phi \in \Psi^{n}_{\alpha}$, $H^{n}_{\alpha}(\phi, j)$ is an element of $\Psi^{m}_{\alpha}$, and
   %For all $n \in \omega$ and $\phi \in \Psi^{n}_{\alpha}$,
   $H^{n}_{\alpha}(\phi, i_{n}) = \phi$.
\end{remark}

We leave it to the reader to verify (by induction on $\alpha$) that if
\begin{itemize}
\item $M$ is a $\tau$-structure,
\item $\alpha$ is an ordinal,
\item $m \leq n$ are elements of $\omega$,
\item $\bar{b} = \langle b_{0},\ldots,b_{n-1} \rangle$ is a sequence of distinct elements of $M$,
\item $j^{*} \colon m \to n$ is an injection,
\item $\bar{a}$ is the sequence $\langle b_{j^{*}(0)},\ldots, b_{j^{*}(m-1)}\rangle$ and
\item $j \in \cF_{m,n}$ is such that $j(x_{p}) = x_{j^{*}(p)}$ for each $p< m$,
\end{itemize}
then $H^{n}_{\alpha}(\phi^{M}_{\bar{b},\alpha},j) = \phi^{M}_{\bar{a},\alpha}$.

%\begin{remark}
%  By condition (\ref{projfact}) of Definition \ref{hdef}, for each ordinal $\alpha$, all $m,n,p \in \omega$ with $m \leq \min\{ n, p\}$, each $\phi \in %\Psi^{n}_{\alpha + 1}$, $\psi \in \Psi^{p}_{\alpha + 1}$,
%each $j \in \cF_{m,n}$ and $k \in \cF_{m,p}$, if $H^{n}_{\alpha + 1}(\phi, j) = H^{p}_{\alpha + 1}(\psi, k)$,
%then $$H^{n+1}_{\alpha}[E(\phi) \times \{j \cup \{(x_{m},y)\} \mid y \in (X_{n+1} \setminus \range(j))\}]$$ is the same as
%$$H^{p+1}_{\alpha}[E(\psi) \times \{k \cup \{(x_{m},y)\} \mid y \in (X_{p+1} \setminus \range(k))\}].$$
%\end{remark}

\begin{remark}\label{tworems}  The following facts can be easily verified by induction on $\alpha$.
\begin{enumerate}
  \item\label{permutevar} For each ordinal $\alpha$, each $n \in \omega$, each $\phi \in \Psi^{n}_{\alpha}$ and each $j \in \cF_{n,n}$, $H^{n}_{\alpha}(\phi, j)$ is the result of replacing each free variable in $\phi$ (i.e., each member of $X_{n}$) with its $j$-preimage.

\item\label{hcompose} For each ordinal $\alpha$, all $m \leq n \leq p$ in $\omega$, all $\phi \in \Psi^{p}_{\alpha}$, all $j \in \cF_{n,p}$ and all $k \in \cF_{m,n}$, $H^{n}_{\alpha}(H^{p}_{\alpha}(\phi, j), k) = H^{p}_{\alpha}(\phi, j \circ k)$.

%\item For each ordinal of the form $\alpha + 1 < \delta$, for all $m \leq n \in \omega$, all $\phi \in \Phi^{n}_{\alpha + 1}$ and all $\psi %\in P(\phi)$, there exist $\psi'$

\end{enumerate}
\end{remark}

The following proposition shows that the vertical and horizontal projection functions commute appropriately.

\begin{prop}\label{proppartone} For all ordinal $\alpha  \leq \beta$, all $m \leq n \in \omega$, all $j \in \cF_{m,n}$, and all $\phi \in \Psi^{n}_{\beta}$, $$V_{\alpha, \beta}(H^{n}_{\beta}(\phi, j)) = H^{n}_{\alpha}(V_{\alpha,\beta}(\phi),j).$$
\end{prop}

\begin{proof}
  When $\alpha = \beta$, both sides are equal to $H^{n}_{\alpha}(\phi, j)$.
  When $\beta = \alpha + 1$, the proposition is part of condition (\ref{projfact}) of Definition \ref{hdef}. When $\beta$ is a limit ordinal, it follows from condition (\ref{pathcond}) of Definition \ref{vdef} and condition (\ref{hsix}) of Definition \ref{hdef}. The remaining cases can be proved by induction on $\beta$, fixing $\alpha$, using the induction hypotheses for the pairs $\alpha, \beta$ and $\beta, \beta + 1$ at successor stages of the form $\beta + 1$.

  %Part (\ref{propparttwo}) follows from part (\ref{proppartone}) of the proposition. Part (\ref{proppartfour}) follows from condition %(\ref{hfour}) of Definition \ref{prodef} and condition (\ref{projfact}) of Definition \ref{prodef}.
\end{proof}

\begin{example}
  Suppose that $\tau$ contains a single binary relation symbol $R$, along with $=$ and the $0$-ary relation symbol $S$. The set $\Psi^{0}_{0}$ then consists of the sentences $S$ and $\neg S$. The set $\Psi^{1}_{0}$ contains four formulas, $S \wedge R(x_{0}, x_{0})$, $S \wedge \neg R(x_{0}, x_{0})$, $\neg S \wedge R(x_{0}, x_{0})$ and $\neg S \wedge \neg R(x_{0}, x_{0})$. Call the first two of these formulas $\psi^{1}_{0}$ and $\phi^{1}_{0}$, respectively.
  The set $\Psi^{2}_{0}$ then contains 32 formulas, for instance, $$S \wedge  \neg R(x_{0}, x_{1}) \wedge \neg R(x_{1}, x_{0}) \wedge R(x_{0}, x_{0}) \wedge R(x_{1}, x_{1}) \wedge x_{0} \neq x_{1}$$ and
  $$S \wedge  \neg R(x_{0}, x_{1}) \wedge \neg R(x_{1}, x_{0}) \wedge R(x_{0}, x_{0}) \wedge \neg R(x_{1}, x_{1}) \wedge x_{0} \neq x_{1}.$$
  Call these formulas $\psi^{2}_{0}$ and $\phi^{2}_{0}$, respectively. Then $$H^{2}_{0}(\psi^{2}_{0}, i_{1}) = \psi^{1}_{0}$$ and $$H^{2}_{0}(\phi^{2}_{0}, \{(x_{0},x_{1})\}) = \phi^{1}_{0},$$ as defined in Definition \ref{hdef}.
  The set $\Psi^{3}_{0}$ then contains $2^{10}$ formulas, including the conjunction of $S$ with every instance of $R(y,z)$ for $y,z \in X_{3}$. In general,
  $\Psi^{n}_{0}$ contains $2^{(n^{2}+1)}$ formulas.

  The set $\Psi^{0}_{1}$ contains the sentences $$S \wedge (\exists x_{0} S \wedge R(x_{0}, x_{0})) \wedge (\forall x_{0} S \wedge R(x_{0}, x_{0}))$$ (omitting one instance each of $\wedge$ and $\vee$, corresponding to a conjunction and a disjunction of of size 1) and
  $$S \wedge ((\exists x_{0} S \wedge R(x_{0}, x_{0})) \wedge (\exists x_{0} S \wedge \neg R(x_{0},x_{0}))) \wedge (\forall x_{0} (S \wedge R(x_{0}, x_{0})) \vee (S \wedge \neg R(x_{0},x_{0}))).$$
   Call these sentences $\psi^{0}_{1}$ and $\phi^{0}_{1}$, respectively. Then $E(\psi^{0}_{1}) = \{\psi^{1}_{0}\}$ and $E(\phi^{0}_{1}) = \{\psi^{1}_{0}, \phi^{1}_{0}\}$, as defined in Definition \ref{pdef}. The set $\Psi^{1}_{1}$ contains the formulas $$\psi^{1}_{0} \wedge (\exists x_{1} \psi^{2}_{0}) \wedge (\forall x_{1}\, x_{1} \neq x_{0} \rightarrow \psi^{2}_{0})$$ and
   $$\psi^{1}_{0} \wedge (\exists x_{1} \phi^{2}_{0}) \wedge (\forall x_{1}\, x_{1} \neq x_{0} \rightarrow \phi^{2}_{0}),$$
   again omitting an instance of each of $\wedge$ and $\vee$ in each formula. Call these formulas $\psi^{1}_{1}$ and $\phi^{1}_{1}$. Then $E(\psi^{1}_{1}) = \{\psi^{2}_{0}\}$, $E(\phi^{1}_{1}) = \{\phi^{2}_{0}\}$, $$V_{0,1}(\psi^{1}_{1}) = V_{0,1}(\phi^{1}_{1}) = \psi^{1}_{0},$$
   %as defined in Definition \ref{vdef},
   $H^{1}_{1}(\psi^{1}_{1}, i_{0}) = \psi_{1}^{0}$ and $H^{1}_{1}(\phi^{1}_{1}, i_{0}) = \phi_{1}^{0}$. Note that the function $H^{1}_{1}$ changes the bound variables (as well as the free variables, when the second coordinate of the input is the empty function).
\end{example}

\section{Scott processes}\label{spsec}

This section introduces the central topic of the paper, the class of Scott processes (for a relational vocabulary $\tau$).

\begin{df}\label{prodef} A \emph{Scott process} is a sequence $\langle \Phi_{\alpha}  :\alpha < \delta\rangle$,
for some ordinal $\delta$ (the \emph{length} of the process), satisfying the following conditions, where for each ordinal $\alpha$ and
each $n \in \omega$, $\Phi^{n}_{\alpha}$ denotes the set $\Phi_{\alpha} \cap \Psi^{n}_{\alpha}$.
\begin{enumerate}
\item {\bf The Formula Conditions}
\begin{enumerate}
\item\label{formone} Each $\Phi_{\alpha}$ is a subset of the corresponding set $\Psi_{\alpha}$.
%and, letting  for each $n \in \omega$, each $\Phi^{n}_{\alpha}$ is closed under permutations of $X_{n}$.

\item\label{econtain} For each ordinal of the form $\alpha + 1 < \delta$, and each $\phi \in \Phi_{\alpha + 1}$, $E(\phi)$ is a subset of $\Phi_{\alpha}$.

\item\label{vfive} For all $\alpha < \beta < \delta$, $\Phi_{\alpha} = V_{\alpha, \beta}[\Phi_{\beta}]$.

\item\label{htwozero} For all $\alpha < \delta$, all $n \in \omega$, all $j \in \cI_{n,n}$ and all $\phi \in \Phi^{n}_{\alpha}$, $H^{n}_{\alpha}(\phi, j) \in
\Phi^{n}_{\alpha}$.

\item\label{htwo} For all $\alpha < \delta$, and all $m < n$ in $\omega$, $\Phi^{m}_{\alpha} = H^{n}_{\alpha}[\Phi^{n}_{\alpha} \times \{i_{m}\}]$.

\end{enumerate}

\item {\bf The Coherence Conditions}
\begin{enumerate}

\item\label{hfour} For each ordinal of the form $\alpha + 1$ below $\delta$, each $n \in \omega$ and each $\phi \in \Phi^{n}_{\alpha + 1}$,
$$E(\phi) = V_{\alpha, \alpha + 1}[\{ \psi \in \Phi^{n+1}_{\alpha + 1} \mid H^{n+1}_{\alpha + 1}(\psi, i_{n}) = \phi\}].$$

%\item\label{proppartthree} For all $\alpha \leq \beta$ such that $\beta + 1 < \delta$, and all $\phi \in \Phi_{\beta + 1}$, $$P(V_{\alpha + %1, \beta + 1}(\phi)) \subseteq V_{\alpha, \beta}[P(\phi)].$$

\item\label{ppath} For all $\alpha < \beta < \delta$, all $n \in \omega$ and all $\phi \in \Phi^{n}_{\beta}$, $$E(V_{\alpha + 1, \beta}(\phi)) \subseteq V_{\alpha, \beta}[\{\psi \in \Phi^{n+1}_{\beta} \mid H^{n+1}_{\beta}(\psi, i_{n}) = \phi\}].$$

\item\label{combine} For all $\alpha < \delta$, $n,m$ in $\omega$, $\phi \in \Phi^{n}_{\alpha}$ and $\psi \in \Phi^{m}_{\alpha}$, there exist $\theta \in \Phi^{n + m}_{\alpha}$ and $j \in \cF_{m, n+m}$ such that $\phi = H^{n+m}_{\alpha}(\theta, i_{n})$ and $\psi = H^{n+m}_{\alpha}(\theta, j)$.

\end{enumerate}

\end{enumerate}
The sets $\Phi_{\alpha}$ are called the \emph{levels} of the Scott process.
\end{df}

\begin{remark}
  Condition (\ref{ppath}) of Definition \ref{prodef} includes the left to right inclusion in condition (\ref{hfour}). We prefer the given formulation of
  condition (\ref{hfour}), as it gives a better sense of the meaning of $E(\phi)$.
\end{remark}

\begin{remark}\label{succrem} Proposition \ref{succcase} shows that equality holds in condition (\ref{ppath}) of Definition \ref{prodef}, for any Scott process,
so that conditions (\ref{hfour}) and (\ref{ppath}) could equivalently be replaced by condition (\ref{ppath}) alone with $=$ in place of $\subseteq$.
\end{remark}

%\begin{remark} Each $\Phi^{n}_{\alpha}$ is closed under permutations of $X_{n}$.
%\end{remark}

\begin{remark}\label{jandi}
  Conditions (\ref{htwozero}) and (\ref{htwo}) of Definition \ref{prodef} combine to give the following: for all $\alpha < \delta$, all $m \leq n$ in $\omega$ and all $j \in \cF_{m,n}$, $\Phi^{m}_{\alpha} = H^{n}_{\alpha}[\Phi^{n}_{\alpha} \times \{j\}]$.
\end{remark}

Proposition \ref{usent} follows from condition (\ref{combine}) of Definition \ref{prodef} and part (\ref{permutevar}) of Remark \ref{tworems}, which implies that $H^{0}_{\alpha}(\phi, i_{0}) = \phi$ for all ordinals $\alpha$ and all $\phi \in \Psi^{0}_{\alpha}$.

\begin{prop}\label{usent} Whenever $\langle \Phi_{\alpha} : \alpha < \delta \rangle$ is a Scott process, $\Phi^{0}_{\alpha}$ has a unique element, for each $\alpha < \delta$.
\end{prop}

\section{Consequences of coherence}\label{ccsec}

In this section we prove some basic facts about Scott processes, primarily about sets of the form $E(\phi)$.
The main result of the section is Proposition \ref{succcase}, which was referred to in Remark \ref{succrem}.
We fix for this section a Scott process $\langle \Phi_{\alpha}  : \alpha < \delta\rangle$.

Proposition \ref{propparttwo} follows from Proposition \ref{proppartone} (i.e., the commutativity of the horizontal and vertical projections).
The failure of the reverse inclusion is witnessed whenever a set of the form $V_{\alpha, \beta}^{-1}[\{\rho\}]$ has more than one member (consider $\theta \in E(\phi_{1}) \setminus E(\phi_{2})$, for $\phi_{1}, \phi_{2} \in V_{\alpha, \beta}^{-1}[\{\rho\}]$).

\begin{prop}\label{propparttwo} For all $\alpha \leq \beta < \delta$, all $m \leq n \in \omega$, all $j \in \cF_{m,n}$, and all $\phi \in \Phi^{m}_{\beta}$, $$V_{\alpha, \beta}[\{ \psi \in \Phi^{n}_{\beta} \mid H^{n}_{\beta}(\psi, j) = \phi\}] \subseteq \{ \theta \in \Phi^{n}_{\alpha} \mid H^{n}_{\alpha}(\theta, j) = V_{\alpha,\beta}(\phi)\}.$$
\end{prop}

%\begin{proof}
%  This follows from Proposition \ref{proppartone}.
%\end{proof}

The right-to-left inclusion in Proposition \ref{proppartfour} says that every one-point extension of a formula $\phi$ at level $\alpha$ is a member of $E(\psi)$, for some $\psi \in V_{\alpha, \alpha + 1}^{-1}[\phi]$. This proposition is used in Remark \ref{fivepointsix}.

\begin{prop}\label{proppartfour} For each ordinal of the form $\alpha + 1$ below $\delta$, each $n \in \omega$ and each $\phi \in \Phi^{n}_{\alpha}$,
$$\bigcup \{ E(\psi) \mid \psi \in V_{\alpha, \alpha+1}^{-1}[\{\phi\}]\} =  \{ \theta \in \Phi^{n+1}_{\alpha} \mid H^{n+1}_{\alpha}(\theta, i_{n}) = \phi\}.$$
\end{prop}

\begin{proof}
The left-to-right inclusion follows from Proposition \ref{propparttwo} and condition (\ref{hfour}) of Definition \ref{prodef}.
The reverse inclusion follows from conditions  (\ref{vfive}) and (\ref{hfour}) of Definition \ref{prodef}, and Proposition \ref{proppartone}.
\end{proof}

Proposition \ref{proppartthree} is the successor case of Proposition \ref{succcase}.

\begin{prop}\label{proppartthree} For all $\alpha \leq \beta$ such that $\beta + 1 < \delta$, and all $\phi \in \Phi_{\beta + 1}$, $$E(V_{\alpha + 1, \beta + 1}(\phi)) = V_{\alpha, \beta}[E(\phi)].$$
\end{prop}

\begin{proof}
Fix $n \in \omega$ such that $\phi \in \Phi^{n}_{\beta + 1}$.
For the forward direction, condition (\ref{ppath}) of Definition \ref{prodef} gives that
$$E(V_{\alpha+1, \beta + 1}(\phi)) \subseteq V_{\alpha, \beta + 1}[\{ \psi \in \Phi^{n+1}_{\beta+1} \mid H^{n+1}_{\beta + 1}(\psi, i_{n}) = \phi\}],$$
which by condition (\ref{vdir0}) of Definition \ref{vdef} is equal to
$$V_{\alpha, \beta}[V_{\beta, \beta + 1}[\{ \psi \in \Phi^{n+1}_{\beta+1} \mid H^{n+1}_{\beta + 1}(\psi, i_{n}) = \phi\}]],$$
which by condition (\ref{hfour}) of Definition \ref{prodef} is equal to $V_{\alpha, \beta}[E(\phi)]$.

For the reverse direction we have from condition (\ref{hfour}) of Definition \ref{prodef} that $V_{\alpha, \beta}[E(\phi)]$ is equal
to $$V_{\alpha, \beta}[V_{\beta, \beta + 1}[\{ \psi \in \Phi^{n+1}_{\beta+1} \mid H^{n+1}_{\beta + 1}(\psi, i_{n}) = \phi\}]],$$
which by Remark \ref{vdir} is equal to
$$V_{\alpha, \alpha+1}[V_{\alpha+1, \beta + 1}[\{ \psi \in \Phi^{n+1}_{\beta+1} \mid H^{n+1}_{\beta + 1}(\psi, i_{n}) = \phi\}]],$$
which by Proposition \ref{propparttwo} is contained in
$$V_{\alpha, \alpha+1}[\{ \theta \in \Phi^{n+1}_{\alpha+1} \mid H^{n+1}_{\alpha + 1}(\theta, i_{n}) = V_{\alpha+1,\beta + 1}(\phi)\}]],$$
which by condition (\ref{hfour}) of Definition \ref{prodef} is equal to $E(V_{\alpha+1,\beta+1}(\phi))$.
\end{proof}

We now show that the reverse inclusion of condition (\ref{ppath}) of Definition \ref{prodef} holds for any Scott process.

\begin{prop}\label{succcase} For all $\alpha <  \beta < \delta$, for all $n \in \omega$ and all $\phi \in \Phi^{n}_{\beta}$, $$E(V_{\alpha + 1, \beta}(\phi)) =  V_{\alpha, \beta}[\{\psi \in \Phi^{n+1}_{\beta} \mid H^{n+1}_{\beta}(\psi, i_{n}) = \phi\}].$$
\end{prop}

\begin{proof}
When $\beta$ is a successor ordinal, this is Proposition \ref{proppartthree}, using condition (\ref{hfour}) of Definition \ref{prodef}.
For any $\beta$, the left-to-right inclusion is condition (\ref{ppath}) of Definition \ref{prodef}.
For the reverse inclusion, $$V_{\alpha, \beta}[\{\psi \in \Phi^{n+1}_{\beta} \mid H^{n+1}_{\beta}(\psi, i_{n}) = \phi\}]$$
is equal to $$V_{\alpha, \alpha+1}[V_{\alpha + 1, \beta}[\{\psi \in \Phi^{n+1}_{\beta} \mid H^{n+1}_{\beta}(\psi, i_{n}) = \phi\}]]$$
by Remark \ref{vdir}, and this is contained in
$$V_{\alpha, \alpha + 1}[\{\psi \in \Phi^{n+1}_{\alpha+1} \mid H^{n+1}_{\alpha+1}(\psi, i_{n}) = V_{\alpha+1,\beta}(\phi)\}],$$
by Proposition \ref{propparttwo}. Finally, this last term is equal to $E(V_{\alpha + 1, \beta}(\phi))$ by condition (\ref{hfour}) of
Definition \ref{prodef}.
\end{proof}

\section{Ranks and Scott sentences}\label{rssec}

The \emph{Scott rank} of a $\tau$-structure $M$ is the least ordinal $\alpha$ such that $V_{\alpha,\alpha+1}$ is injective on $\Phi_{\alpha + 1}(M)$ (see \cite{Hodges, Marker}, which use different terminology, of course, to define the same thing).
If $\alpha$ is the Scott rank of $M$, then $V_{\beta, \beta + 1}$ injective on $\Phi_{\beta + 1}(M)$ for all $\beta \geq \alpha$ as well. Proposition \ref{leastfixed} below verifies that Scott processes have the same property. We isolate the successor step of the proof as a separate proposition (the second part of the proposition is used in Remark \ref{ecover}).

\begin{prop}\label{leastfixedalt}  Let $\beta$ be an ordinal, and let $\langle \Phi_{\alpha} : \alpha \leq \beta + 2 \rangle$ be a Scott process. If
$\phi$ is an element of $\Phi_{\beta + 1}$, then each of the following conditions implies that $V_{\beta + 1, \beta + 2}^{-1}[\{\phi\}] \cap \Phi_{\beta + 2}$ is a singleton.
\begin{enumerate}
\item\label{lf1} $V_{\beta,\beta+1}^{-1}[\{\psi\}] \cap \Phi_{\beta + 1}$ is a singleton for each $\psi \in E(\phi)$.
\item\label{lf2} There exists a $\psi \in E(\phi)$ such that $V_{\beta, \beta + 2}^{-1}[\{\psi\}] \cap \Phi_{\beta + 2}$ is a singleton.
\end{enumerate}
\end{prop}

\begin{proof}
  Suppose that $\phi' \in \Phi_{\beta + 2}$ is such that $V_{\beta + 1, \beta + 2}(\phi') = \phi$. Assuming the first condition, by Proposition \ref{proppartthree}, $V_{\beta, \beta+1}[E(\phi')] = E(\phi)$. Since $V_{\beta,\beta+1}^{-1}[\{\psi\}] \cap \Phi_{\beta + 1}$ is a singleton for each $\psi \in E(\phi)$, this implies that $E(\phi') = V_{\beta, \beta+1}^{-1}[E(\phi)] \cap \Phi_{\beta + 1}$, which uniquely determines $\phi'$.

  For the second part, let $\psi'$ be the unique member of $V_{\beta, \beta + 2}^{-1}[\{\psi\}] \cap \Phi_{\beta + 2}$. Since $\psi \in E(\phi)$, $V_{\beta + 1, \beta + 2}(\psi')$ is a member of $E(\phi')$, by Proposition \ref{proppartthree}.
  Let $n \in \omega$ be such that $\phi \in \Phi^{n}_{\beta + 1}$.
  Then $\phi' = H^{n+1}(\psi', i_{n})$, by part (\ref{hfour}) of Definition \ref{prodef}.
\end{proof}

The following is a consequence of part (\ref{lf1}) of Proposition \ref{leastfixedalt}.

\begin{cor}\label{fixeddown}
   Let $\beta$ be an ordinal, and let $\langle \Phi_{\alpha} : \alpha \leq \beta + 2 \rangle$ be a Scott process. Suppose that $n \in \omega$ is such that
   $V_{\beta, \beta + 1}$ is injective on $\Phi^{n+1}_{\beta + 1}$. Then $V_{\beta + 1, \beta + 2}$ is injective on $\Phi^{n}_{\beta + 2}$.
\end{cor}

%\begin{prop}Let $\beta$ be an ordinal, and let $\langle \Phi_{\alpha} : \alpha \leq \beta + 2 \rangle$ be a Scott process. Fix $n \in \omega$ and suppose that %$V_{\beta, \beta + 1} \restrict \Phi^{n+1}_{\beta + 1}$ is injective. Then $V_{\beta + 1, \beta + 2} \restrict \Phi^{n}_{\beta  +2}$ is injective.
%\end{prop}

%\begin{proof}
%  Fix formulas $\phi_{0}$ and $\phi_{1}$ in $\Phi^{n}_{\beta + 2}$ with
%  $V_{\beta + 1, \beta + 2}(\phi_{0}) = V_{\beta + 1, \beta + 2}(\phi_{1})$ (which we will call $\phi'$). By Proposition \ref{proppartthree},
%  $V_{\beta, \beta + 1}[E(\phi_{0})] = V_{\beta, \beta + 1}[E(\phi_{1})]$. Since $V_{\beta, \beta + 1} \restrict \Phi^{n+1}_{\beta +1}$ is injective,
%  $E(\phi_{0}) = E(\phi_{1})$, which implies that $\phi_{0} = \phi_{1}$.
%\end{proof}

\begin{remark}
  It is natural to ask whether part (\ref{lf1}) of Proposition \ref{leastfixedalt} has a converse, in the sense that if $\langle \Phi_{\alpha} : \alpha \leq \beta + 1 \rangle$ is a Scott process and $\phi \in \Phi_{\beta + 1}$ and $\psi \in E(\phi)$ are such that $V_{\beta,\beta+1}^{-1}[\{\psi\}]$ has at least two members then there must exist a set $\Phi_{\beta + 2}$ such that $\langle \Phi_{\alpha} : \alpha \leq \beta + 2\rangle$ is a Scott process and $V_{\beta + 1, \beta + 2}^{-1}[\{\phi\}]$ is not a singleton. This is not the case in general, however, as by Proposition \ref{usent}, each function of the form $V_{\alpha, \alpha + 1}\restrict \Phi^{0}_{\alpha + 1}$ is always injective.
\end{remark}

%\begin{remark}
%  Suppose that $\mathcal{P} = \langle \Phi_{\alpha} : \alpha \leq \delta \rangle$ is a Scott process. Fix $n \in \omega$ and $\phi \in \Phi^{n}_{\delta}$.
%  Let $X$ be a subset of $\{ \psi \in \Phi^{n+1}_{\delta} \mid H(\psi, i_{n}) = \phi\}$ with the property that $V_{\alpha, \delta}^{-1}[\{\theta\}] \cap X$ is %nonempty for each $\alpha < \delta$ and each
%  $\theta \in E(V_{\alpha + 1, \delta}(\phi))$. Then there is a Scott process \langle \Phi_{\alpha} : \alpha \leq \delta + 1\rangle extending $\mathcal{P}$ and a %formula $\phi' \in \Phi^{n}_{\alpha + 1}$ such that $E(\phi') = X$ (which implies that $V_{\delta, \delta + 1}(\phi') = \phi$). This can be proved by a recursive %construction where one chooses a sequence of formulas $\langle \phi_{m} : m \geq n\rangle$ with $\phi_{n} = \phi$, with suitable sets $\langle X_{n} : n \geq %m\rangle$ with $X_{n} = X$. OOPS, YOU NEED TO WORRY ABOUT COMPATIBALITY
%\end{remark}

%The following question asks for a sort of converse for Proposition \ref{leastfixedalt}.

%\begin{ques}\label{leastfixedalt}  Let $\beta$ be an ordinal, and let $\langle \Phi_{\alpha} : \alpha \leq \beta + 1 \rangle$ be a Scott process. Suppose that %%$\phi \in \Phi_{\beta + 1}$ and $\psi \in E(\phi)$ are such that $V_{\beta,\beta+1}^{-1}[\{\psi\}]$ has at least two members. Must there exist a set %$\Phi_{\beta + 2}$ such that $\langle \Phi_{\alpha} : \alpha \leq \beta + 2\rangle$ is a Scott process and $V_{\beta + 1, \beta + 2}^{-1}[\{\phi\}]$ is not a %singleton?
%\end{ques}

\begin{prop}\label{leastfixed} If $\langle \Phi_{\alpha} : \alpha < \delta \rangle$ is a Scott process, $\beta < \gamma$ are ordinals with $\gamma + 1 < \delta$, and $V_{\beta,\beta + 1} \restrict \Phi_{\beta +1}$ is injective, then $V_{\gamma, \gamma + 1} \restrict \Phi_{\gamma + 1}$ is injective.
\end{prop}

\begin{proof}
  Letting $\eta$ be such that $\gamma = \beta + \eta$, we prove the proposition by induction on $\eta$, for all $\beta$ and $\delta$ simultaneously.
  Applying the induction hypotheses, the limit case follows from Remark \ref{vdir}, and the successor case follows from part (\ref{lf1}) of Proposition \ref{leastfixedalt} (and also from Corollary \ref{fixeddown}).
\end{proof}

%  the case where $\eta = 1$.
%  Let us assume then that $\gamma = \beta + 1$. Suppose that $n \in \omega$ and $\phi_{0}$ and $\phi_{1}$ are formulas in $\Phi^{n}_{\gamma + 1}$ with
%  $V_{\gamma, \gamma + 1}(\phi_{0}) = V_{\gamma, \gamma + 1}(\phi_{1})$ (which we will call $\phi'$). By Proposition \ref{proppartthree},
%  $V_{\beta, \beta + 1}[E(\phi_{0})] = V_{\beta, \beta + 1}[E(\phi_{1})]$. Since $V_{\beta, \beta + 1} \restrict \Phi_{\beta +1}$ is injective,
%  $E(\phi_{0}) = E(\phi_{1})$, which implies that $\phi_{0} = \phi_{1}$.
%\end{proof}

%We record an alternate version of Proposition \ref{leastfixed}. The proof of Proposition \ref{leastfixedalt} is the same as the proof of the $\eta = 1$
%case of the proof of Proposition \ref{leastfixed}, but we repeat it for clarity.

\begin{df} The \emph{rank} of a Scott process $\langle \Phi_{\alpha} : \alpha < \delta \rangle$ is the least $\beta$ such that $V_{\beta, \beta + 1} \restrict \Phi_{\beta + 1}$ is injective, if such a $\beta$ exists, and undefined otherwise.
\end{df}

The rank of (any suitably long set-sized initial segment of) the Scott process of a $\tau$-structure $M$ is the same then as the Scott rank of $M$.

\begin{remark}
Suppose that $\beta$ and $\gamma$ are ordinals, and $n \in \omega$ is such that $\gamma > \beta + n$. Suppose that $\langle \Phi_{\alpha} : \alpha < \gamma \rangle$ is a Scott process, and that $V_{\beta, \beta+1}$ is injective on $\Phi^{m}_{\beta + 1}$, for all $m > n$ in $\omega$. By Corollary \ref{fixeddown}, the rank of $\langle \Phi_{\alpha} : \alpha < \gamma \rangle$ is at most $\beta + n$ (since each $\Phi^{0}_{\alpha}$ is a singleton, $V_{\alpha, \alpha + 1} \restrict \Phi^{0}_{\alpha + 1}$ is injective for all $\alpha$).
\end{remark}

In the following definition, $j$ can equivalently be replaced with $i_{n}$, by condition (\ref{htwozero}) of Definition \ref{prodef}.

\begin{df}\label{beyonddef} Let $\beta$ and $\gamma$ be ordinals such that $\gamma > \beta + 1$, and let $$\langle \Phi_{\alpha} : \alpha < \gamma \rangle$$ be a Scott process.
Let $n$ be an element of $\omega$, and let $\phi$ be an element of $\Phi^{n}_{\beta}$. We say that the Scott process $\langle \Phi_{\alpha} : \alpha < \gamma \rangle$ is \emph{injective beyond} $\phi$ if for all $m \in \omega \setminus n$, all $j \in \cI_{n,m}$
%(equivalently, just $i_{n}$)
and all $\psi \in \Phi^{m}_{\beta}$ such that $\phi = H^{m}_{\beta}(\psi, j)$, $V_{\beta, \beta+ 1}^{-1}[\{\psi\}] \cap \Phi_{\beta + 1}$ is a singleton.
\end{df}

\begin{remark}\label{preoneextension}
  Let $\beta < \delta < \gamma$ be ordinals, and let $\langle \Phi_{\alpha} : \alpha \leq \gamma \rangle$ be a Scott process. Let $n$ be an element of $\omega$, and let $\phi \in \Phi^{n}_{\beta}$ be such that $\langle \Phi_{\alpha} : \alpha < \gamma \rangle$ is injective beyond $\phi$. Then for all $m \in \omega \setminus n$, all $j \in \cI_{n,m}$ and all $\psi \in \Phi^{m}_{\beta}$ such that $\phi = H^{m}_{\beta}(\psi, j)$, $V_{\beta, \gamma}^{-1}[\{\psi\}] \cap \Phi_{\gamma}$ is a singleton. This follows from part (\ref{lf1}) of Proposition \ref{leastfixedalt}.
  It follows that $\langle \Phi_{\alpha} : \alpha \leq \gamma \rangle$ is injective beyond the unique member of $V^{-1}_{\beta, \delta}[\{\phi\}]$.
\end{remark}

\begin{remark}\label{oneextension}
  Let $\beta$ be an ordinal, and $n$ an element of $\omega$. Suppose that $$\langle \Phi_{\alpha} : \alpha \leq \beta + 1 \rangle$$ is a Scott process, and that $\phi \in \Phi^{n}_{\beta}$ is such that $\langle \Phi_{\alpha} : \alpha \leq \beta + 1 \rangle$ is injective beyond $\phi$. The proof of Scott's Isomorphism Theorem (Theorem 2.4.15 of \cite{Marker}; using $\bar{a}$ in place of $\emptyset$ at stage 0) shows that for any two countable $\tau$-structures $M$ and $N$ whose Scott processes agree with $\langle \Phi_{\alpha} : \alpha \leq \beta + 1 \rangle$ though level $\beta + 1$, if $\bar{a}$ is an $n$-tuple from $M$ and $\bar{b}$ is an $n$-tuple from $N$, each satisfying $\phi$ in their respective models, then there is an isomorphism of $M$ and $N$ sending $\bar{a}$ to $\bar{b}$.
  Alternately, one can show that for each ordinal $\gamma > \beta + 1$, there is a unique Scott process of length $\gamma$ extending $\beta$, using either
  Remark \ref{preoneextension} or Proposition \ref{fdetermine}.
  %(note that if $\langle \Phi_{\alpha} : \alpha \leq \beta + 1 \rangle$ is injective beyond $\phi$ then for any Scott process $\langle \Phi_{\alpha} : \alpha %\leq \gamma \rangle$ extending $\langle \Phi_{\alpha} : \alpha \leq \beta + 1 \rangle$ and any ordinal $\eta \in [\beta + 1, \gamma)$, and the corresponding %extension of $\phi$, by part (\ref{lf1}) of Proposition \ref{leastfixedalt}). This can also be shown using Proposition \ref{fdetermine}.
\end{remark}

\begin{remark}\label{oneextension2}
  In the situation of Definition \ref{beyonddef}, $\langle \Phi_{\alpha} : \alpha \leq \beta + 1 \rangle$ need not have rank $\beta$. To see this, consider
  the Scott process of a countably infinite undirected graph $G$ consisting of an infinite set of nodes which are not connected to anything, and another infinite set of nodes which are all connected to each other, but not to themselves. The formula in $\Phi^{2}_{0}(G)$ corresponding to a connected pair has the property of $\phi$ in Remark \ref{oneextension}, but the Scott rank of $G$ is $1$, not $0$, since the unique member of $\Phi^{1}_{0}(G)$ has two successors in $\Phi^{1}_{1}(G)$.
\end{remark}

The following definition is inspired by Remarks \ref{oneextension} and \ref{oneextension2}.

\begin{df}
  The \emph{pre-rank} of a Scott process $\langle \Phi_{\alpha} : \alpha < \beta \rangle$ is the least $\gamma \leq \beta$ such that for all ordinals $\eta>\gamma$,
  there exists a unique Scott process of length $\eta$ extending $\langle \Phi_{\alpha} : \alpha < \gamma \rangle$ (if such a $\gamma$ exists). The Scott \emph{pre-rank} of a $\tau$-structure is the pre-rank of the sufficiently long initial segments of its Scott process.
\end{df}

The pre-rank of a Scott process is at most its rank, and Remark \ref{oneextension2} shows that it can be smaller. By Proposition \ref{laterbound}, if a Scott process has countable length, and all of its levels are countable, then its rank is at most $\omega$ more than its pre-rank.  Proposition \ref{prerankbound} gives a tighter bound in the situation of Definition \ref{beyonddef}.

%The proposition will be used in Section \ref{ellsec}.

%In contrast to Remark \ref{oneextension2}, we have the following observation.

\begin{prop}\label{prerankbound}
  Let $\beta$ be an ordinal, and $n$ an element of $\omega$. Suppose that $\langle \Phi_{\alpha} : \alpha \leq \beta + \omega \rangle$ is a Scott process, and that $\phi \in \Phi^{n}_{\beta}$ is such that $\langle \Phi_{\alpha} : \alpha \leq \beta + \omega \rangle$ is injective beyond $\phi$.
  %for all $m \in \omega \setminus n$, all $j \in \cI_{n,m}$ and all $\psi \in \Phi^{m}_{\beta}$ such that $\phi = H^{m}_{\beta}(\psi, j)$, $V_{\beta, \beta+ %1}^{-1}[\{\psi\}] \cap \Phi_{\beta + 1}$ is a singleton.
  Then $\langle \Phi_{\alpha} : \alpha \leq \beta + \omega \rangle$ has rank at most $\beta + n$.
\end{prop}

\begin{proof}
  By Remark \ref{preoneextension}, for each $m \in \omega\setminus n$, $V^{-1}_{\beta, \beta + \omega}[\{\psi\}] \cap \Phi_{\beta + \omega}$ is a singleton for each $\psi \in \Phi^{m}_{\beta}$ such that $\phi = H^{m}_{\beta}(\psi, j)$ for some $j \in \cI_{n,m}$. Let $\Upsilon$ be the set of $\psi \in \Phi_{\beta}$ for which $V^{-1}_{\beta, \beta + n + 1}[\{\psi\}] \cap \Phi_{\beta + \omega}$ is a singleton. We show by induction on $p \leq n$ that if $q \in \omega$ and $\theta \in \Phi^{q}_{\beta + p}$ is such that $\theta = H^{q + p}_{\beta + p}(\rho, i_{q})$ for some $\rho \in \Phi^{q+p}_{\beta + p}$ such that $V_{\beta, \beta + p}(\rho) \in \Upsilon$, then $V^{-1}_{\beta + p, \beta + n + 1}[\{\theta\}]$ is a singleton. For $p = 0$ this follows from the definition of $\Upsilon$. The induction step from $p$ to $p+1$ (for some $\theta \in \Phi^{q}_{\beta + p + 1}$, for some $q \in \omega$) follows from part (\ref{lf1}) of Proposition \ref{leastfixedalt}, applied to $$V_{\beta + p, \beta + p + 1}(H^{q+p+1}_{\beta + p + 1}(\rho, i_{q + 1})),$$ where $\rho \in \Phi^{q+p+1}_{\beta + p+1}$ is such that $\theta = H^{q + p + 1}_{\beta + p + 1}(\rho, i_{q})$ and $V_{\beta, \beta + p+1}(\rho) \in \Upsilon$. Finally, this statement for $n = p$ implies the proposition, by condition (\ref{combine}) of Definition \ref{prodef}.
\end{proof}

%Given ordinals $\beta, \delta$ with $\beta + 1 < \delta$, a Scott process $\langle \Phi_{\alpha} : \alpha < \delta \rangle$ is \emph{trivial above} $\beta$ if %$V_{\beta, \beta + 1} \restrict \Phi_{\beta + 1}$ is injective.
%\end{df}

%  Towards showing that $E(\phi_{0}) = E(\phi_{1})$, fix a formula $\psi \in E(\phi_{0})$. By ..., $V_{\beta, \beta + 1}(\psi) \in E(\phi')$. By condition %(\ref{ppath}) of Definition \ref{prodef} (with $\beta$ as $\alpha$ and $\gamma + 1$ as $\beta$), there exists a $\theta \in \Phi^{n+1}_{\gamma + 1}$ such that %$H_{\gamma + 1}(\theta, i_{n}) = \phi_{1}$ and
%  $V_{\beta, \gamma + 1}(\theta) = V_{\beta, \beta + 1}(\psi)$. Since $V_{\beta, \beta + 1}$ is injective, this means that $V_{\gamma, \gamma + 1}(\theta) =
%  \psi$, which by condition (\ref{hfour}) of Definition \ref{prodef} implies that $\psi \in E(\phi_{1})$.

%\begin{remark} Condition (\ref{ppath}) of Definition \ref{prodef} implies that
%\end{remark}

\begin{df}\label{maxdef} Given an ordinal $\delta$ and a set $\Phi \subseteq \Psi_{\delta}$, the \emph{maximal completion} of $\Phi$ is
%the sequence $\langle \Phi_{\alpha} :\alpha \leq \delta + 1\rangle$, where $\Phi_{\delta + 1}$ is
the set of
$\phi \in \Psi_{\delta + 1}$ such that for some $n \in \omega$ and some $\phi' \in \Phi \cap \Psi^{n}_{\delta}$, $V_{\delta, \delta + 1}(\phi) = \phi'$, and $$E(\phi) = \{ \psi \in \Phi \cap \Psi^{n+1}_{\delta} \mid H^{n+1}_{\delta}(\psi, i_{n}) = \phi'\}.$$
\end{df}

The extension of a Scott process of successor length by the maximal completion of its last level may not be a Scott process (see Proposition \ref{canmax} below).
%If $\langle \Phi_{\alpha} : \alpha < \delta \rangle$ is a Scott process, and $\beta$ is minimal such that $\beta + 1 < \delta$ and $\Phi_{\beta + 1}$ is the %maximal completion of $\Phi_{\beta}$, then $\langle \Phi_{\alpha} : \alpha < \delta \rangle$ corresponds to a model of Scott rank $\beta$.

%\begin{df}\label{ssdef} Suppose that $\langle \Phi_{\alpha}  :\alpha  < \delta\rangle$ is a Scott process, and that $\beta < \delta$ is minimal
%such that $V_{\beta, \beta + 1} \restrict \Phi_{\beta + 1}$ is injective.
%%for each $\phi \in \Phi_{\delta}$ there is a unique $\psi \in \Phi_{\delta + 1}$ such that $V_{\delta, \delta+1}(\phi, \psi)$,
%Then
%the \emph{Scott sentence} corresponding to $\langle \Phi_{\alpha}  :\alpha  < \delta\rangle$ is the
%sentence $$\phi \wedge \bigwedge_{n \in \omega\setminus\{0\}}\forall x_{0}\ldots\forall x_{n-1}\bigwedge\{ V_{\beta, \beta + 1}(\psi) \rightarrow \psi : \psi \in %\Phi^{n}_{\beta + 1}\},$$ where $\phi$ is the unique element of
%$\Phi^{0}_{\beta + 1}$.
%\end{df}

\begin{remark}\label{fivepointsix} By Proposition \ref{proppartfour}, if $\langle \Phi_{\alpha} : \alpha < \delta \rangle$ is a Scott process, and $\beta$ is an ordinal such that $\beta + 1< \delta$, then $V_{\beta, \beta + 1} \restrict \Phi_{\beta + 1}$ is injective if and only if $\Phi_{\beta + 1}$ is the maximal completion of
$\Phi_{\beta}$.
\end{remark}

%\pagebreak

%\begin{df}\label{unidef} Given an ordinal $\delta$, a set $\Phi \subseteq \Psi_{\delta}$ is \emph{unified} if for all $m < n < \omega$, $\phi \in \Phi \cap %\Psi^{n}_{\delta}$
%%, $j \in \mathcal{F}_{m+1,n}$
%and $k \in \mathcal{F}_{m,n}$ such that
%  $H^{n}_{\delta}(\phi, i_{m}) = H^{n}_{\delta}(\phi, k)$, there exist $\psi \in \Phi \cap \Psi^{n+1}_{\delta}$ and $y \in X_{n+1} \setminus \range(k)$ such %that
%  $H_{\delta}(\psi, i_{n}) = \phi$ and $$H_{\delta}(\psi, i_{m+1}) = H_{\delta}(\psi, k \cup \{(x_{m}, y)\}).$$
%\end{df}

The following definition describes the situation in which no formula $\phi$ has incompatible horizontal extensions.

\begin{df}\label{unidef} Given an ordinal $\delta$, a set $\Phi \subseteq \Psi_{\delta}$ satisfies the \emph{amalgamation property} (or \emph{amalgamates}) if for all $m < n \in \omega$, $\phi \in \Phi \cap \Psi^{m+1}_{\delta}$, and $\psi \in \Phi \cap \Psi^{n}_{\delta}$ such that $H^{m+1}(\phi,i_{m}) = H^{n}_{\delta}(\psi, i_{m})$, there exist $\theta \in \Phi \cap \Phi^{n+1}_{\delta}$ and $y \in X_{n+1} \setminus X_{m}$ such that $H^{m+1}(\theta, i_{m} \cup \{(x_{m}, y)\}) = \phi$ and $H^{m+1}(\theta, i_{n}) = \psi$.
\end{df}

\begin{remark}\label{unidef2} Given an ordinal $\delta$ and a set $\Phi \subseteq \Phi_{\delta}$ satisfying condition (\ref{htwozero}) of Definition \ref{prodef} (i.e., closure under the functions $H^{n}_{\delta}$ ($n \in \omega$)), the amalgamation property for a set $\Phi \subseteq \Psi_{\delta}$ is equivalent to the statement that for all $m \leq n \in \omega$, $\phi \in \Phi \cap \Psi^{m}_{\delta}$, $j \in \mathcal{F}_{m,n}$ and $\psi \in \Phi \cap \Psi^{n}_{\delta}$ such that $\phi = H^{n}_{\delta}(\psi, j)$,
$$\{ \theta \in \Phi \cap \Psi^{m+1}_{\delta} \mid H^{m+1}_{\delta}(\theta, i_{m}) = \phi\}$$ is the same as
$$H^{n+1}_{\delta}[\{ \rho \in \Phi \cap \Psi^{n+1}_{\delta} \mid H^{n+1}_{\delta}(\rho, i_{n}) = \psi\} \times \{j \cup \{(x_{m},y)\} \mid y \in (X_{n+1} \setminus \range(j))\}].$$ This follows immediately from the definitions (using part (\ref{projfact}) of Definition \ref{hdef}).
\end{remark}

%Proposition \ref{amalgamation} gives a reformulation of unity.

\begin{remark} The set in the second displayed formula in Definition \ref{unidef} is always contained in the set in the first, by part (\ref{hcompose}) of Remark \ref{tworems}.
\end{remark}

%\begin{remark}\label{unirem} Let $\delta$ be an ordinal, and let $\Phi$ be a subset of $\Psi_{\delta}$ such that for each $n \in \omega$, $H^{n}_{\delta}[(\Phi %\cap \Phi^{n}_{\delta}) \times \cF_{n,n}] \subseteq \Phi$ (i.e., $\Phi$ is closed under free-variable permutations). Then $\Phi$ is unified if and and only if %for all $m < n < \omega$, $\phi \in \Phi \cap \Psi^{n}_{\delta}$ $j \in \mathcal{F}_{m+1,n}$
%and $k \in \mathcal{F}_{m,n}$ such that
%  $H^{n}_{\delta}(\phi, j \restrict X_{m}) = H^{n}_{\delta}(\phi, k)$, there exist $\psi \in \Phi \cap \Psi^{n+1}_{\delta}$ and $y \in X_{n+1} \setminus %\range(k)$ such that
%  $H^{n+1}_{\delta}(\psi, i_{n}) = \phi$ and $$H^{n+1}_{\delta}(\psi, j) = H^{n+1}_{\delta}(\psi, k \cup \{(x_{m}, y)\}).$$
%\end{remark}

\begin{prop}\label{canmax}
  The extension of a nonempty Scott process of nonlimit length by the maximal completion of its last level induces a Scott process if and only if
  its last level amalgamates.
\end{prop}

\begin{proof} Let $\langle \Phi_{\alpha} : \alpha \leq \delta \rangle$ be a Scott process. Conditions (\ref{formone})-(\ref{vfive}) of Definition
\ref{prodef} are always satisfied by the extension by the maximal completion. The other conditions depend on whether the functions $H^{n}_{\delta + 1}$ ($n \in \omega$) lift the actions of the functions $H^{n}_{\delta}$ ($n \in \omega$), i.e., whether whenever $n \in \omega$, $\psi \in \Phi^{n}_{\delta}$ and $\psi'$ is the unique member of $V_{\delta, \delta+ 1}^{-1}[\{\psi\}]$ in the maximal completion of $\Phi_{\delta}$, $$V_{\delta, \delta + 1}[H^{n}_{\delta+1}(\psi', j)] =
H^{n}_{\delta}(\psi, j).$$ Comparing the condition (\ref{projfact}) of Definition \ref{hdef} with Definition \ref{maxdef} shows that is exactly the statement that $\Phi_{\delta}$ amalgamates as expressed in Remark \ref{unidef2}.
\end{proof}

We conclude this section by giving a restatement of the amalgamation property which will be useful in Section \ref{arbcardsec}. A failure of amalgamation gives a counterexample to Proposition \ref{amalgamation} with $n = m+1$.

%See Definition \ref{unidef} for the definition of a unified level of a Scott process.

%The following proposition can be proved by induction on $n-m$, using Remark \ref{unirem}.

%\pagebreak

%\begin{prop}\label{uniext} Let $\delta$ be an ordinal, let $\Phi \subseteq \Psi_{\delta}$ be unified and suppose that for each $n \in \omega$, %$H^{n}_{\delta}[(\Phi \cap \Phi^{n}_{\delta}) \times \cF_{n,n}] \subseteq \Phi$.
%Let $m < n \leq p$ be elements of $\omega$, let $\phi \in \Phi \cap \Psi^{p}_{\delta}$
%$j \in \mathcal{F}_{n,p}$
%and $k \in \mathcal{F}_{m,p}$ be such that
%  $H^{p}_{\delta}(\phi, j \restrict X_{m}) = H^{p}_{\delta}(\phi, k)$. Then there exist $\psi \in \Phi \cap \Psi^{p+(n-m)}_{\delta}$ and $k' \in \cF_{p,p+(n-m)}$ %extending $k$ such that
%  $$\phi = H_{\delta}(\psi, i_{p}) = H_{\delta}(\psi, k').$$
%\end{prop}

%Rephrasing this, we get the following.

%\pagebreak

\begin{prop}\label{amalgamation}
Suppose that $\langle \Phi_{\alpha} : \alpha \leq \delta \rangle$ is Scott process whose last level amalgamates, and that $m,n,p \in \omega$ are such that
$m \leq \min\{n,p\}$.
Suppose now that
%$\phi \in \Phi^{m}_{\delta}$,
$j \in \cF_{m,n}$, $k \in \cF_{m,p}$, $\psi \in \Phi^{n}_{\delta}$ and $\theta \in \Phi^{p}_{\delta}$ are such that $$H^{n}_{\delta}(\psi, j) = H^{p}_{\delta}(\theta, k).$$ Then there exist $q \in \omega \setminus \max\{n,p\}$, a formula $\rho \in \Phi^{q}_{\delta}$ and functions $j' \in \cF_{n,q}$ and $k' \in \cF_{p,q}$ such that $X_{q} = \range(j') \cup \range(k')$, $j' \circ j = k' \circ k$, $\psi = H^{q}_{\delta}(\rho,j')$ and $\theta = H^{q}_{\delta}(\rho, k')$.
\end{prop}

\begin{proof}
  Fixing $m$ and $p$, we prove the proposition by induction on $n$. If $n =m$, then we can let $q = p$, $\rho = \theta$, $k' = i_{p}$ and $j' = k \circ j^{-1}$.
  %If $n-m = 1$, let $j_{0} \in \cF_{n,n}$ be such that $\range(j_{0} \circ j) = X_{m}$.
  Suppose that the proposition holds for $n$.
  Let $j \in \cF_{m,n+1}$, $k \in \cF_{m,p}$, $\psi \in \Phi^{n+1}_{\delta}$ and $\theta \in \Phi^{p}_{\delta}$ be such that $H^{n+1}_{\delta}(\psi, j) = H^{p}_{\delta}(\theta, k)$. Let $f \in\cF_{n+1,n+1}$ be the identity function if $x_{n} \not\in\range(j)$; otherwise, fix $n'$ such that $x_{n'} \not\in\range(j)$ and let $f$ map $x_{n}$ and $x_{n'}$ to each other and fix the rest of $X_{n+1}$. Let $\psi_{0} = H^{n+1}_{\delta}(\psi, f)$. Then
  $x_{n} \not \in \range(f \circ j)$.  Let $\psi_{1} = H^{n+1}_{\delta}(\psi_{0}, i_{n})$. By the second part of Remark \ref{tworems},
 \begin{center} $\begin{array} {lcl} H^{n}_{\delta}(\psi_{1}, f \circ j) & = & H^{n}_{\delta}(H^{n+1}_{\delta}(\psi_{0},i_{n}),f \circ j) \\ & = & H^{n+1}_{\delta}(\psi_{0},i_{n} \circ (f \circ j))\\ & = & H^{n+1}_{\delta}(\psi_{0}, f \circ j) \\ & = & H^{n+1}_{\delta}(H^{n+1}_{\delta}(\psi_{0}, f), j) \\
 &=&  H^{n+1}_{\delta}(\psi, j)\\ & =& H^{p}_{\delta}(\theta, k).\end{array}$\end{center}
 Applying the induction hypothesis to
  $f \circ j$, $k$, $\psi_{1}$ and $\theta$, we get $q_{0} \in \omega \setminus \max\{n,p\}$, a formula $\rho_{0} \in \Phi^{q_{0}}_{\delta}$ and functions $j_{0} \in \cF_{n,q_{0}}$ and $k' \in \cF_{p,q_{0}}$ such that $$X_{q_{0}} = \range(j_{0}) \cup \range(k'),$$ $j_{0} \circ (f \circ j) = k' \circ k$, $\psi_{1} = H^{q_{0}}_{\delta}(\rho_{0},j_{0})$ and $\theta = H^{q_{0}}_{\delta}(\rho_{0}, k')$.

  Suppose first that there exists a $y \in X_{q_{0}} \setminus \range(j_{0})$ such that $$\psi_{0} = H^{q_{0}}_{\delta}(\rho_{0}, j_{0} \cup \{(x_{n}, y)\}).$$
  Then $q_{0}$, $\rho_{0}$ and $k'$ are as desired. If $f = i_{n+1}$, then we can let $j' = j_{0} \cup \{(x_{n}, y)\}$ and we are done. Otherwise, let $j'$ send $x_{n'}$ to $y$,  $x_{n}$ to $j_{0}(x_{n'})$ and every other member of $X_{n}$ to the same place that $j_{0}$ does (i.e., let $j' = (j_{0} \cup \{(x_{n}, y)\}) \circ f$).
  Then $j' \circ j = k' \circ k$, and
  \begin{center}$\begin{array}{lcl} \psi &=& H^{n+1}_{\delta}(\psi_{0}, f)\\
  &=& H^{n+1}_{\delta}(H^{q}_{\delta}(\rho, j_{0} \cup \{(x_{n}, y)\}), f)\\
  &=& H^{q}_{\delta}(\rho, (j_{0} \cup \{(x_{n}, y)\})\circ f)\\
  &=& H^{q}_{\delta}(\rho, j'),
  \end{array}$\end{center}
  as desired.

  Finally, suppose that there is no such $y \in X_{q_{0}} \setminus \range(j_{0})$. Putting together the amalgamation property of $\Phi$ and the equation $\psi_{1} = H^{n+1}_{\delta + 1}(\psi_{0}, i_{n})$, we get that there exist a formula $\rho \in \Phi^{q_{0}+1}_{\delta}$ such that $H^{q_{0}+1}_{\delta}(\rho, i_{q_{0}}) = \rho_{0}$ and a $y \in X_{q_{0} + 1} \setminus \range(j_{0})$ such that $H^{q_{0}+1}_{\delta}(\rho, j_{0} \cup \{(x_{n}, y)\}) = \psi_{0}$. Then $k'$, $\rho$, and $q = q_{0} + 1$ are as desired. If $f = i_{n+1}$, then we can let $j' = j_{0} \cup \{(x_{n}, y)\}$, and we are done.
  Otherwise, as above,
  %let $j'$ send $x_{n'}$ to $y$,  $x_{n}$ to $j_{0}(x_{n'})$ and every other member of $X_{n}$ to the same place that $j_{0}$ does (i.e., let
  let $j' = (j_{0} \cup \{(x_{n}, y)\}) \circ f$.
  Then again $j' \circ j = k' \circ k$ and $\psi = H^{q}_{\delta}(\rho, j')$,
 % \begin{center}$\begin{array}{lcl} \psi &=& H^{n+1}_{\delta}(\psi_{0}, f)\\
 % &=& H^{n+1}_{\delta}(H^{q}_{\delta}(\rho, j_{0} \cup \{(x_{n}, y)\}), f)\\
 % &=& H^{q}_{\delta}(\rho, (j_{0} \cup \{(x_{n}, y)\})\circ f)\\
 % &=& H^{q}_{\delta}(\rho, j'),
 % \end{array}$\end{center}
  as desired.
\end{proof}

\section{Building countable models}\label{threadsec}

In this section we show that any Scott process of successor length has a countable model if its last level is countable. This in turn implies that such a sequence can be extended to any given ordinal length (although the rank of the Scott process of length $\omega_{1}$ corresponding to a countable model is countable).
%Fix for this section a Scott process $\langle \Phi_{\alpha}, V_{\alpha,\beta}, H_{\alpha}  :\alpha  \leq \beta \leq \delta\rangle$.
%The material in this section is generalized in Section \ref{arbcardsec}.

\begin{df}\label{threaddef} Given an ordinal $\beta$, and a countable set $\Phi \subseteq \Psi_{\beta}$, a \emph{thread} through $\Phi$ is a set of
formulas $\{ \phi_{n} : n \in \omega \} \subseteq \Phi$ such that
  \begin{enumerate}
  \item for all $n \in \omega$, $\phi_{n} \in \Psi^{n}_{\beta}$;
  \item\label{allin} for all $m < n$ in $\omega$, $\phi_{m} = H^{n}_{\beta}(\phi_{n}, i_{m})$;
  %\item\label{pickup} for all $m \in \omega$ and all $\psi \in \Phi \cap \Psi^{m}_{\beta}$, there exist $n \in \omega$ and $j \in \cF_{m,n}$ such that
  %$\psi = H^{n}_{\beta}(\phi_{n}, j)$;
  \item\label{goodlift} for all $m \in \omega$, all $\alpha < \beta$, and all $\psi \in E(V_{\alpha + 1, \beta}(\phi_{m}))$, there exist an $n \in \omega \setminus (m+1)$ and a $y \in X_{n} \setminus X_{m}$ such that $\psi = V_{\alpha, \beta}(H^{n}_{\beta}(\phi_{n}, i_{m} \cup \{(x_{m}, y)\}))$.
  \end{enumerate}
\end{df}

%\begin{remark} Condition (\ref{allin}) of Definition \ref{threaddef} implies that each $\phi_{n}$ is in $\Psi^{n}_{\beta}$.
%\end{remark}

\begin{remark} If $\beta$ is a successor ordinal, condition (\ref{goodlift}) of Definition \ref{threaddef} is equivalent to the restriction of the
condition to the case where $\alpha = \beta - 1$. This follows from Proposition \ref{proppartthree}. Similarly, condition (\ref{goodlift}) of Definition \ref{threaddef} is equivalent to the restriction of the condition to the set of $\alpha$ in any cofinal subset of $\beta$.
%\footnote{Does condition (\ref{pickup}) follow from  condition (\ref{goodlift}) in the limit case?}
\end{remark}

\begin{remark}\label{cofrestr}
  Suppose that $\langle \Phi_{\alpha} : \alpha \leq \delta \rangle$ is a Scott process, and $\beta < \delta$ is such that $V_{\beta, \delta} \restrict \Phi_{\delta}$ is injective. Then the $V_{\beta+1, \delta}$-preimage of a thread through $\Phi_{\beta+1}$ is a thread through $\Phi_{\delta}$. This follows from Remark \ref{vone}, Proposition \ref{proppartone} and Proposition \ref{proppartthree}.
\end{remark}

\begin{prop}\label{threadcon} If $\langle \Phi_{\alpha} : \alpha \leq \delta \rangle$ is a Scott process with $\Phi_{\delta}$ countable, then there exists a thread through $\Phi_{\delta}$.
\end{prop}

\begin{proof}
  %By condition (\ref{combine}) of Definition \ref{prodef}, $\{\phi_{n} : n \in \omega \}$ can be chosen to meet conditions (\ref{allin})
  %and (\ref{pickup})
  %of Definition \ref{threaddef}.
  By Remark \ref{cofrestr}, if suffices to consider the cases where $\delta$ is either a successor ordinal or an ordinal of cofinality $\omega$. Let $A$ be $\{\delta - 1\}$ in the case where $\delta$ is a successor ordinal, and a countable cofinal subset of $\delta$ otherwise.
  We choose the formulas  $\phi_{n}$ recursively, meeting instances of condition (\ref{goodlift}) of Definition \ref{threaddef} for $\alpha \in A$ while satisfying condition (\ref{allin}). Note that
  $\phi_{0}$ is the unique element of $\Phi^{0}_{\delta}$. To satisfy an instance of condition (\ref{goodlift}), we need to see that if $m \leq n$ are in $\omega$, $\alpha \in A$, $\phi_{n}$ has been chosen, and $\psi \in E(V_{\alpha+1,\beta}(\phi_{m}))$ is not equal to $V_{\alpha, \beta}(H^{n}_{\beta}(\phi_{n}, i_{m} \cup \{(x_{m}, y)\}))$ for any $y \in X_{n} \setminus X_{m}$, then
  $\phi_{n+1}$ can be chosen so that $$\psi = V_{\alpha, \beta}(H^{n+1}_{\beta}(\phi_{n+1}, i_{m} \cup \{(x_{m}, x_{n})\}))$$ (since $\Phi_{\delta}$ is countable, the set of such formulas $\psi$ is also countable). The existence of such a $\phi_{n+1}$ follows from condition (\ref{projfact}) of Definition \ref{hdef} applied to $V_{\alpha + 1, \beta}(\phi_{n})$ and $i_{m}$, giving a $\theta \in E(V_{\alpha + 1, \beta}(\phi_{n}))$ such that $$H^{n+1}_{\alpha}(\theta, i_{m} \cup \{(x_{m}, x_{n})\}) = \psi,$$ followed by condition (\ref{ppath}) of Definition \ref{prodef} applied to $\phi_{n}$, giving $\phi_{n+1}$ as desired.
\end{proof}

\begin{thrm}\label{existmodel} Given a Scott process $\langle \Phi_{\alpha} : \alpha \leq \delta \rangle$ with $\Phi_{\delta}$ countable, a thread $\langle \phi_{n} : n \in \omega \rangle$ through $\Phi_{\delta}$ and a set $C = \{ c_{n} : n \in \omega \}$, there is a $\tau$-structure with domain $C$ in which each tuple $\langle c_{m} : m < n\rangle$ satisfies $\phi_{n}$.
\end{thrm}

\begin{proof} Let each tuple $\langle c_{m} : m < n\rangle$ satisfy all the atomic formulas indicated by $V_{0,\delta}(\phi_{n})$. We show by induction on $\alpha$ that each tuple $\langle c_{m} : m < n \rangle$ satisfies the formula $V_{\alpha, \delta}(\phi_{n})$. This follows immediately for limit stages. For the induction step from $\alpha$ to $\alpha + 1$, $\langle c_{m} : m < n \rangle$ satisfies $V_{\alpha + 1, \delta}(\phi_{n})$ if and only if $$E(V_{\alpha + 1, \delta}(\phi_{n})) = V_{\alpha, \delta}[\{ H^{p}_{\delta}(\phi_{p}, i_{n} \cup \{ (x_{n}, y)\}) : p \in \omega \setminus (n+1), y \in X_{p} \setminus X_{n}\}].$$
That is, checking that $\langle c_{m} : m < n \rangle$
satisfies $V_{\alpha+1,\delta}(\phi_{n})$ means showing that the left side of the
equality is the set of formulas from $\Phi^{n+1}_{\alpha}$ satisfied by extensions of
$\langle c_{m} : m < n \rangle$ by one point, which by the induction hypothesis is what the right side
is.
The left-to-right containment follows from condition (\ref{goodlift}) of Definition \ref{threaddef}. For the other direction, note first that by Proposition \ref{succcase},
$$E(V_{\alpha + 1, \delta}(\phi_{n})) =  V_{\alpha, \delta}[\{\theta \in \Phi^{n+1}_{\delta} \mid H^{n+1}_{\delta}(\theta, i_{n}) = \phi_{n}\}].$$
That $$\{ H^{p}_{\delta}(\phi_{p}, i_{n} \cup \{ (x_{n}, y)\}) : p \in \omega \setminus (n+1), y \in X_{p} \setminus X_{n}\}$$ is contained in  $\{\theta \in \Phi^{n+1}_{\delta} \mid H^{n+1}_{\delta}(\theta, i_{n}) = \phi_{n}\}$ follows from the assumption that $\phi_{n} = H^{p}_{\delta}(\phi_{p}, i_{n})$.
\end{proof}

%To see that $$V_{\alpha, \delta}[\{\theta \in \Phi^{n+1}_{\delta} \mid H_{\delta}(\theta, i_{n}) = \phi_{n}\}]$$ is contained in
%$$V_{\alpha, \delta}[\{ H_{\delta}(\phi_{p}, i_{n} \cup \{ (x_{n}, y)\}) : p \in \omega \setminus (n+1), y \in X_{p} \setminus X_{n}\}],$$ note first that
%$V_{\alpha, \delta}[\{\theta \in \Phi^{n+1}_{\delta} \mid H_{\delta}(\theta, i_{n}) = \phi_{n}\}]$ is equal to
%$E(V_{\alpha+1, \delta}(\phi_{n}))$ by condition (\ref{ppath}) of Definition \ref{prodef}.
%\end{proof}

\begin{df}\label{cthreaddef}
  Given an ordinal $\beta$, and a countable set $\Phi \subseteq \Psi_{\beta}$, a thread  $\{ \phi_{n} : n \in \omega \}$ through $\Phi$ is \emph{complete} if for all $m \in \omega$ and all $\psi \in \Phi \cap \Psi^{m}_{\beta}$, there exist $n \in \omega$ and $j \in \cF_{m,n}$ such that $\psi = H^{n}_{\beta}(\phi_{n}, j)$.
\end{df}

\begin{remark}\label{cthreadrem}
  The thread through $\Phi_{\delta}$ given by Proposition \ref{threadcon} induces a model (via Theorem \ref{existmodel}) whose Scott process contains $\langle \Phi_{\alpha} : \alpha < \delta\rangle$, and for which the $\delta$-th level of the corresponding Scott process is contained in the given $\Phi_{\delta}$. The $\delta$-th level is equal to $\Phi_{\delta}$ if and only if the thread is complete.
  Condition (\ref{combine}) of Definition \ref{prodef} implies that one can add stages to the construction in Proposition \ref{threadcon} to produce a complete thread.
\end{remark}

\begin{df} A $\tau$-structure $M$ is a \emph{model} of a Scott process $\langle \Phi_{\alpha} : \alpha < \delta \rangle$ if $\Phi_{\alpha} = \Phi_{\alpha}(M)$ for all $\alpha <  \delta$.
\end{df}

Proposition \ref{threadcon}, Theorem \ref{existmodel} and Remark \ref{cthreadrem} give the following.

\begin{thrm}\label{ctblehasmodel}
Every Scott process $\langle \Phi_{\alpha} : \alpha \leq \delta \rangle$ with $\Phi_{\delta}$ countable has a countable model.
\end{thrm}

\begin{remark}
  %The material in Section \ref{ellsec} can be used to show that if $\mathcal{P} = \langle \Phi_{\alpha} : \alpha < \delta \rangle$ is a Scott process, and $A$ is %a countable cofinal subset of $\delta$, then $\mathcal{P}$ can be extended by one level to a new Scott process with the new level having cardinality %$\sum_{\alpha \in A}|\Phi_{\alpha}|$.
  Theorem \ref{limitext} gives a stronger version of
  Theorem \ref{ctblehasmodel}, showing that every Scott process with all
  levels countable (and possibly of limit length) has a model.
\end{remark}

%Since any $\tau$-structure induces a class-length Scott process, Proposition \ref{oneext} below follows from Theorem \ref{existmodel}.
%We give an alternate proof, as it illustrates the issues in extending a Scott process of successor length by one level.

\begin{remark} If $\langle \Phi_{\alpha} : \alpha \leq \delta \rangle$ is a Scott process, $\gamma < \delta$ and $\{ \phi_{n} : n \in \omega \}$ is a thread through $\Phi_{\delta}$, then $\{ V_{\gamma, \delta}(\phi_{n}) : n \in \omega \}$ is a thread through $\Phi_{\gamma}$ (this follows from Proposition \ref{proppartone}). This thread induces (as in the proof of Theorem \ref{existmodel}) the same class-length Scott process as $\{\phi_{n} : n \in \omega \}$.
\end{remark}

We insert here an argument for constructing a pair of models.
The issue of extending this theorem to uncountable models is discussed in Remark \ref{addthisrem}.
% which will be used in the proof of Theorem \ref{Harringtonplus}.

\begin{thrm}\label{addthis}
  Let $\gamma$ be a countable ordinal, and suppose that $\langle \Phi_{\beta} : \beta \leq \gamma \rangle$ is a Scott process with $\Phi_{\gamma}$ countable.
  Let $\Phi^{*}$ be a subset of $\Phi_{\gamma}$ such that the extension of $\langle \Phi_{\beta} : \beta \leq\gamma \rangle$ by $\Phi^{*}$ is also a Scott process. Then there exists $\tau$-structures $M$ and $N$ such that $M$ is a substructure of $N$, $N$ is a model of $\langle \Phi_{\beta} : \beta \leq \gamma \rangle$
  and $M$ is a model of the extension of $\langle \Phi_{\beta} : < \gamma \rangle$ by $\Phi^{*}$.
\end{thrm}

\begin{proof}
  By Theorem \ref{existmodel}, it suffices to find a complete thread $\langle \phi_{n} : n \in \omega \rangle$ through $\Phi_{\gamma}$ and a infinite set $Y \subseteq \omega$ such that, letting (for each $n \in \omega$)
  \begin{itemize}
  \item $j_{n}$ be the order preserving map from $X_{n}$ to the first $n$ elements of the set $\{x_{m} : m \in Y\}$,
  \item $k_{n}$ be the least element of $\omega$ such that $|Y \cap k_{n}| = n$,
  \end{itemize}
  $\langle H^{n}_{\gamma}(\phi_{k_{n}}, j_{n}) : n \in \omega\rangle$ is a complete thread through $\Phi^{*}$.

  A construction of such a pair $\langle \phi_{n} : n \in \omega \rangle$, $Y$ can be carried out in essentially the same manner as the proof of Theorem \ref{threadcon} (as above, completeness can be achieved using Condition (\ref{combine}) of Definition \ref{prodef}). The construction of $\langle \phi_{n} : n \in \omega \rangle$ is exactly the same, except with stages inserted to ensure that the formulas $H^{n}_{\gamma}(\phi_{k_{n}}, j_{n})$ are as desired. That is, if $k_{m} \leq n$ (as determined so far), $\alpha \in A$ (as in the proof of Theorem \ref{threadcon}), $\phi_{n}$ has been chosen, and $\psi \in E(V_{\alpha+1,\beta}(H^{n}_{\gamma}(\phi_{k_{m}}, j_{m}))$ is not equal to $V_{\alpha, \beta}(H^{n}_{\beta}(\phi_{n}, j_{m} \cup \{(x_{m}, y)\}))$ for any $x_{p} \in X_{n} \setminus X_{m}$ with $p$ already chosen to be in $Y$, then
  $\phi_{n+1}$ can be chosen so that $$\psi = V_{\alpha, \beta}(H^{n+1}_{\beta}(\phi_{n+1}, j_{m} \cup \{(x_{m}, x_{n})\})).$$ As in the proof of Theorem \ref{threadcon}, the existence of such a $\phi_{n+1}$ follows from condition (\ref{projfact}) of Definition \ref{hdef} applied to $V_{\alpha + 1, \beta}(\phi_{n})$ and $j_{m}$, giving a $\theta \in E(V_{\alpha + 1, \beta}(\phi_{n}))$ such that $$H^{n+1}_{\alpha}(\theta, j_{m} \cup \{(x_{m}, x_{n})\}) = \psi,$$ followed by condition (\ref{ppath}) of Definition \ref{prodef} applied to $\phi_{n}$, giving $\phi_{n+1}$ as desired. We then put $n$ into $Y$ and continue the construction.
\end{proof}

\section{Models of cardinality $\aleph_{1}$}\label{arbcardsec}

In this section we show how to build models for Scott processes of length a successor ordinal, under the assumption that the last level of the process amalgamates and has cardinality at most $\aleph_{1}$.

Given two finite sets of ordinals $a \subseteq b$ with $a = \{\alpha_{0},\ldots,\alpha_{n-1}\}$ (listed in increasing order), let $j_{a,b}$ be the function $j$ in
$\cI_{n,|b|}$ such that $j(x_{m}) = x_{|b \cap \alpha_{m}|}$ for all $m < n$.

In the case $\kappa = \omega$, the following definition is essentially the same as Definition \ref{threaddef}, as the formulas $\{ \phi_{n} : n \in \omega\}$ of the weaving then satisfy Definition \ref{threaddef}.

\begin{df}\label{threaddef2}
Suppose that $\delta$ is an ordinal, $\kappa$ is an infinite cardinal and $\Phi$ is a subset of $\Psi_{\delta}$ of cardinality $\kappa$.
 %is a Scott process and that $|\Phi_{\delta}| = \kappa$.
 A \emph{weaving} through $\Phi$ is a set of formulas $\{ \phi_{a} : a \in [\kappa]^{\less\omega} \} \subseteq \Phi$ such that the following hold.
\begin{enumerate}
\item\label{basic} each $\phi_{a} \in \Psi^{|a|}_{\delta}$.
\item\label{allin2} For all $a \subseteq b \in [\kappa]^{\less\omega}$, $\phi_{a} = H^{|b|}_{\delta}(\phi_{b}, j_{a,b})$.
%\item\label{pickup2} for all $n \in \omega$ and all $\psi \in \Phi \cap \Psi^{n}_{\delta}$, there exist $a \in [\kappa]^{n}$ and $j \in \cF_{n,n}$ such that
%  $\psi = H^{n}_{\delta}(\phi_{a}, j)$;
\item\label{goodlift2} for all $a \in [\kappa]^{\less\omega}$, all $\alpha < \delta$, and all $\psi \in E(V_{\alpha + 1, \delta}(\phi_{a}))$, there exist a $b \in [\kappa]^{|a|+1}$ containing $a$ and a $y \in X_{|b|} \setminus \range(j_{a,b})$ such that $$\psi = V_{\alpha, \delta}(H^{|a|+1}_{\delta}(\phi_{b}, j_{a,b} \cup \{(x_{|a|}, y)\})).$$
\end{enumerate}
\end{df}

%If $\Phi$ is the last level of a Scott process, and $\Phi$ is unified, then the condition in Proposition \ref{betterpick} can always be met, i.e., there always %exists a formula which can serve as $\phi_{a \cup \{\beta\}}$. This follows from Proposition \ref{amalgamation}.
%\end{remark}

The proof of Theorem \ref{existmodel2} is an adaptation of the proof of Theorem \ref{existmodel}.

\begin{thrm}\label{existmodel2} Given a Scott process $\langle \Phi_{\alpha} : \alpha \leq \delta \rangle$ with $\delta$ and $\Phi_{\delta}$ of cardinality $\kappa$, a weaving $\langle \phi_{a} : a \in [\kappa]^{\less\omega} \rangle$ through $\Phi_{\delta}$ and a set $C = \{ c_{\alpha} : \alpha < \kappa \}$, there is a $\tau$-structure with domain $C$ in which, for each $a \in [\kappa]^{\less\omega}$, the tuple $\langle c_{\alpha} : \alpha \in a\rangle$ satisfies $\phi_{a}$.
\end{thrm}

\begin{proof} For each $a \in [\kappa]^{\less\omega}$, let the tuple $\langle c_{\alpha} : \alpha \in a\rangle$ satisfy all the atomic formulas indicated by $V_{0,\delta}(\phi_{a})$. We show by induction on $\beta < \delta$ that each tuple $\langle c_{\alpha} : \alpha \in a \rangle$ satisfies the formula $V_{\beta, \delta}(\phi_{a})$. This follows immediately for limit stages. For the induction step from $\beta$ to $\beta + 1$, $\langle c_{\alpha} : \alpha \in a \rangle$ satisfies $V_{\beta + 1, \delta}(\phi_{a})$ if and only if $E(V_{\beta + 1, \delta}(\phi_{a}))$ is equal to $$V_{\beta, \delta}[\{ H^{|b|}_{\delta}(\phi_{b}, j_{a,b} \cup \{ (x_{|a|}, y)\}) : a \subseteq b \in [\kappa]^{\less\omega}, y \in X_{|b|} \setminus \range(j_{a,b})\}].$$
The left-to-right containment follows from condition (\ref{goodlift2}) of Definition \ref{threaddef2}. For the other direction, note first that by Proposition \ref{succcase},
$$E(V_{\beta + 1, \delta}(\phi_{a})) =  V_{\beta, \delta}[\{\theta \in \Phi^{|a|+1}_{\delta} \mid H^{|a|+1}_{\delta}(\theta, i_{|a|}) = \phi_{a}\}].$$
That $$\{ H^{|b|}_{\delta}(\phi_{b}, j_{a,b} \cup \{ (x_{|a|}, y)\}) : a \subseteq b \in [\kappa]^{\less\omega}, y \in X_{|b|} \setminus \range(j_{a,b})\}$$ is contained in  $\{\theta \in \Phi^{|a|+1}_{\delta} \mid H^{|a|+1}_{\delta}(\theta, i_{|a|}) = \phi_{a}\}$ follows from the assumption that $\phi_{a} = H^{|b|}_{\delta}(\phi_{b}, j_{a,b})$.
\end{proof}

\begin{df}\label{cweavedef}
  Suppose that $\delta$ is an ordinal, $\kappa$ is an infinite cardinal and $\Phi$ is a subset of $\Psi_{\delta}$ of cardinality $\kappa$.
 A weaving $\{ \phi_{a} : a \in [\kappa]^{\less\omega} \} \subseteq \Phi$ through $\Phi$ is \emph{complete} if for all $n \in \omega$ and all $\psi \in \Phi \cap \Psi^{n}_{\delta}$, there exist $a \in [\kappa]^{n}$ and $j \in \cF_{n,n}$ such that  $\psi = H^{n}_{\delta}(\phi_{a}, j)$;
\end{df}

\begin{remark}\label{cweaverem}
As in Remark \ref{cthreadrem}, given a Scott process $\langle \Phi_{\alpha} : \alpha \leq \delta \rangle$ and a weaving through $\Phi_{\delta}$, the proof of
Theorem \ref{existmodel2} gives a model whose Scott process has $\langle \Phi_{\alpha} : \alpha < \delta\rangle$ as an initial segment, and for which the $\delta$-th level of its Scott process is contained in the given $\Phi_{\delta}$. The $\delta$-th level is equal to $\Phi_{\delta}$ if and only if the weaving is complete.
\end{remark}

It remains to find a complete weaving through a Scott process $\langle \Phi_{\alpha} : \alpha \leq \delta \rangle$, assuming that $\Phi_{\delta}$ amalgamates and has cardinality $\aleph_{1}$.

%Proposition \ref{betterpick} follows from condition (\ref{ppath}) of Definition \ref{prodef}.

\begin{df}\label{sweavingdef}
Suppose that $\delta$ is an ordinal, $\kappa$ is an infinite cardinal and $\Phi$ is a subset of $\Psi_{\delta}$ of cardinality $\kappa$.
A \emph{strong weaving} through $\Phi$ is a set $$\{ \phi_{a} : a \in [\kappa]^{\less\omega} \} \subseteq \Phi$$ satisfying conditions (\ref{basic}) and (\ref{allin2}) of Definition \ref{threaddef2} plus the following condition:
%Condition
%(\ref{pickup2}) and
%(\ref{goodlift2}) of Definition \ref{threaddef2} is implied by the following condition :
for all $a \in [\kappa]^{\less\omega}$, and all $\psi \in \Phi \cap \Psi^{|a|+1}_{\delta}$ such that $H^{|a|+1}_{\delta}(\psi, i_{|a|}) = \phi_{a}$, there exist a $b \in [\kappa]^{|a|+1}$ containing $a$ and a $y \in X_{|b|} \setminus \range(j_{a,b})$ such that $$\psi = H^{|a|+1}_{\delta}(\phi_{b}, j_{a,b} \cup \{(x_{|a|}, y)\})).$$
\end{df}

\begin{remark} In condition (\ref{goodlift2}) of Definition \ref{threaddef2} and in Definition \ref{sweavingdef}, the variable $y$ is in fact the unique
member of $X_{|b|} \setminus \range(j_{a,b})$.
\end{remark}

\begin{prop}\label{betterpick}
Suppose that $\delta$ is an ordinal, $\kappa$ is an infinite cardinal and $\Phi$ is a subset of $\Psi_{\delta}$ of cardinality $\kappa$.
A strong weaving through $\Phi$ is both a weaving and complete.
\end{prop}

%Suppose that $\{ \phi_{a} : a \in [\kappa]^{\less\omega} \} \subseteq \Phi$ satisfies conditions (\ref{basic}) and (\ref{allin2}) of Definition \ref{threaddef2}. %Then $\{ \phi_{a} : a \in [\kappa]^{\less\omega} \}$ is a complete weaving through $\Phi$ if it satisfies the following condition:
%%Condition
%%(\ref{pickup2}) and
%%(\ref{goodlift2}) of Definition \ref{threaddef2} is implied by the following condition :
%for all $a \in [\kappa]^{\less\omega}$, and all $\psi \in \Phi \cap \Psi^{|a|+1}_{\delta}$ such that $H^{|a|+1}_{\delta}(\psi, i_{|a|}) = \phi_{a}$, there exist %a $b \in [\kappa]^{|a|+1}$ containing $a$ and a $y \in X_{|b|} \setminus \range(j_{a,b})$ such that $$\psi = H^{|a|+1}_{\delta}(\phi_{b}, j_{a,b} \cup \{(x_{m}, %y)\})).$$
%\end{prop}

\begin{proof}
  That a strong weaving satisfies condition (\ref{goodlift2}) of Definition \ref{threaddef2} follows from condition (\ref{ppath}) of Definition \ref{prodef}.
  Completeness for formulas in $\Phi \cap \Psi^{n}_{\delta}$ follows by induction on $n$.
  %from applying the condition successively to each formula $H^{n}_{\delta}(\psi,i_{m})$, for $m \leq n$.
\end{proof}

%for all $m \leq |a|$, all $j\in \cF_{m,|a|}$, and all $\psi \in \Phi \cap \Psi^{m+1}_{\delta}$ such that $H_{\delta}(\phi_{a}, j) = H_{\delta}(\psi, i_{m})$, %there exist $\beta \in \kappa \setminus a$
%%$\theta \in \Phi^{|a|+1}_{\delta}$
%and $y \in X_{|a| + 1} \setminus \range(j_{a,a\cup\{\beta\}})$ such that
%%$H_{\delta}(\phi_{a \cup \{\beta\}}, j_{a,a\cup\{\beta\}) = \phi_{a}$ and
%$$H_{\delta}(\phi_{a \cup \{\beta\}}, (j_{a,a\cup\{\beta\}} \circ j) \cup \{(x_{m}, y)\}) = \psi.$$
%\end{lem}

%\begin{proof}
%In the case of condition (\ref{pickup2}), this follows from applying the condition successively to each $H_{\delta}(\psi,i_{p})$, $p \leq m$.
%\end{proof}

A subset $S$ of a collection $C$ of sets is $\subseteq$-\emph{cofinal} in $C$ if every member of $C$ is contained in a member of $S$.

\begin{prop}\label{cofstrong}
Suppose that $\langle \Phi_{\alpha} : \alpha \leq \delta \rangle$ is a Scott process such that $\Phi_{\delta}$ amalgamates. Let $\mathcal{W} = \{ \phi_{a} : a \in [\kappa]^{\less\omega} \}$ be a subset of $\Phi$ satisfying conditions (\ref{basic}) and (\ref{allin2}) of Definition \ref{threaddef2}, such that the set of $a \in [\kappa]^{\less\omega}$ for which the condition in Definition \ref{sweavingdef} is satisfied is $\subseteq$-cofinal in $[\kappa]^{\less\omega}$. Then $\mathcal{W}$ is a strong weaving.
\end{prop}

\begin{proof} Suppose that we have $a \subseteq b \in [\kappa]^{\less\omega}$, and that the condition in Definition \ref{sweavingdef} holds for $b$.
%and formulas $\phi_{a}$ and $\phi_{c}$ satisfying conditions (\ref{basic}) and (\ref{allin2}) of Definition \ref{threaddef2} (in the proof of Theorem %\ref{existweaving}, $c$ will be a grounded set, and $a$ will be a set for which we are trying to complete a task given by the condition in Proposition %\ref{betterpick}).
Suppose that $\psi \in \Phi^{|a|+1}_{\delta}$ is such that $H^{|a|+1}_{\delta}(\psi, i_{|a|}) = \phi_{a}$.
%as in the statement of Proposition \ref{betterpick}.
By Proposition \ref{amalgamation}, there is formula $\theta \in \Phi^{|b|+1}$ such that $H^{|b|+1}_{\delta}(\theta, i_{|b|}) = \phi_{b}$
and $H^{|b|+1}_{\delta}(\theta, j_{a,b} \cup \{(x_{|a|},x_{|b|})\}) = \psi$. Then there exist a $\beta \in \kappa \setminus b$ a $y \in X_{|b|+1} \setminus \range(j_{b,b \cup \{\beta\}})$ such that $$\theta = H^{|b|+1}_{\delta}(\phi_{b \cup \{\beta\}}, j_{b,b \cup \{\beta\}} \cup \{(x_{|b|}, y)\}),$$ which implies
that $$\psi = H^{|a|+1}_{\delta}(\phi_{a \cup \{\beta\}}, j_{a,a \cup \{\beta\}} \cup \{(x_{|a|}, y)\}),$$ for
$y$ the unique element of $X_{|a|+1} \setminus \range(j_{a,a \cup \{\beta\}})$.
\end{proof}

%Proposition \ref{amalgamation} then implies that there exist an ordinal $\beta \in \kappa \setminus a$ and a choice of $\phi_{c \cup \{\beta\}}$ satisfying %conditions (\ref{basic}) and (\ref{allin2}) of Definition \ref{threaddef2} such that, letting $b = a \cup \{\beta\}$, $b$ and $\phi_{b} =
%H^{|c \cup \{\beta\}|}_{\delta}(\phi_{c \cup \{\beta\}}, j_{b, c \cup \{\beta\}})$ satisfy the conclusion of Proposition \ref{betterpick}. In other words,
%since $\Phi_{\delta}$ is unified, the condition in Proposition \ref{betterpick} is satisfied for all $a$ if it is satisfied for $\subseteq$-cofinally many $a$.
%\end{proof}

\begin{prop}\label{existweaving} If $\langle \Phi_{\alpha} : \alpha \leq \delta \rangle$ is a Scott process such that $\Phi_{\delta}$ amalgamates and has cardinality $\aleph_{1}$, then there is a strong weaving through $\Phi_{\delta}$.
\end{prop}

%Finally, we want to  this one gets that if $\langle \Phi_{\alpha} : \alpha \leq \delta \rangle$ is a unified Scott process, there is a $\tau$-structure which is %a model of $\Phi_{\delta}$. Building the structure involves picking formulas $\phi_{a}$ for each finite subset $a$ of some ordinal.

\begin{proof} We recursively pick suitable formulas $\phi_{a}$, for $a \in [\omega_{1}]^{\less\omega}$. To begin with, let $\phi_{n}$ $(n \in \omega)$ be any elements of $\Phi_{\delta}$ with the property that $H^{n}_{\delta}(\phi_{n}, i_{m}) = \phi_{m}$, for all $m \leq n < \omega$. Suppose now that we have $\alpha < \omega_{1}$ and that $\phi_{a}$ has been chosen for each finite subset of $\alpha$ (note that a choice of $\phi_{a}$ determines a choice of $\phi_{b}$ for each subset of $b$, where $a$ is a finite subset of $\omega_{1}$). Following some bookkeeping, we fix the least pair $a, \psi$ as in Definition \ref{sweavingdef} for which the corresponding condition has not been met, and let $\phi_{a \cup \{\alpha\}}$ be this $\psi$. Fixing a bijection $\pi \colon \omega \to (\alpha \setminus a)$, we now successively choose the formulas $\phi_{a \cup \{\alpha\} \cup \pi[n]}$. For each positive $n$, the choice of $\phi_{a \cup \{\alpha\} \cup \pi[n]}$ requires amalgamating $\phi_{a \cup \{\alpha\} \cup \pi[n-1]}$ with $\phi_{a \cup \pi[n]}$, which have already been chosen. The fact that $\Phi_{\delta}$ amalgamates (via Proposition \ref{amalgamation}) implies that there exists a suitable choice of $\phi_{a \cup \{\alpha\} \cup \pi[n]}$. Since $\phi_{a \cup \{\pi(n-1)\}}$ did not satisfy third condition of Definition \ref{sweavingdef} with respect to $a$ and $\psi$, this choice of $\phi_{a \cup \{\alpha\} \cup \pi[n]}$ does not require identifying $\pi(n-1)$ and $\alpha$. Proceeding in this fashion completes the construction of the desired strong weaving.
\end{proof}

Putting together Theorem \ref{existmodel2} with Propositions \ref{betterpick} and \ref{existweaving}, we have the following.

\begin{thrm}\label{omega1model}
  If $\langle \Phi_{\alpha} : \alpha \leq \delta \rangle$ is a Scott process, $\Phi_{\delta}$ amalgamates and $|\Phi_{\delta}| \leq \aleph_{1}$, then
  $\langle \Phi_{\alpha} : \alpha \leq \delta \rangle$ has a model.
\end{thrm}

\begin{remark}
  One difference between the construction in this section and the construction in Section \ref{threadsec} is that in Proposition \ref{threadcon} there were (in some cases) many options for the thread produced, as every $\tau$-structure whose Scott process extends $\langle \Phi_{\alpha} : \alpha \leq \delta \rangle$ is given by a thread through $\Phi_{\delta}$. In this section, with uncountably many tasks to complete, we need to rely on the fact that $\Phi_{\delta}$ amalgamates, which means that we construct models of Scott rank $\delta$ only. The assumption that $\Phi_{\delta}$ amalgamates also enables the simplification given by Proposition \ref{betterpick}.
\end{remark}

\begin{remark}\label{noomega2}
  One might naturally try to adapt the proof of Theorem \ref{existmodel2} to build a model of size $\aleph_{2}$ by assigning a formula from $\Phi_{\delta}$ to each finite tuple from $\omega_{2}$. Doing this in the manner of the proof of Theorem \ref{existmodel2}, one finds oneself with an uncountable $\alpha < \omega_{2}$ such that formulas have been assigned for all finite subsets of $\alpha$, but not for $\{ \alpha\}$. Choosing formulas for all finite subsets of $\alpha + 1$, one comes to a point where, for some countably infinite $B \subseteq \alpha$, formulas have been chosen for all sets of the form $\{ \alpha \} \cup b$, for $b$ a finite subset of $B$. Then, for some $\beta \in \alpha \setminus B$, one would like to chose a formula for some finite superset $c$ of $\{\alpha, \beta\}$ intersecting $B$.
  Finally, consider $\gamma \in B \setminus c$. We have at this point that formulas have been chosen for $\{\alpha, \gamma\}$, $\{\beta, \gamma\}$ and $c$, but not for $\{\alpha, \beta, \gamma\}$, and our assumptions do not give us suitable choice for $\{\alpha, \beta, \gamma\}$ that extends the choices already made.
  One can naturally define a notion of 3-amalgamation such that this construction could succeed under the assumption that this property holds.
\end{remark}

\begin{remark}\label{addthisrem}
  The natural attempt to combine the proofs of Theorem \ref{addthis} and Proposition \ref{existweaving} to produce a version of Theorem \ref{addthis} for models of size $\aleph_{1}$ runs into a problem similar to the one in Remark \ref{noomega2}. In this case, we have a Scott process $\langle \Phi_{\alpha} : \alpha \leq \beta \rangle$, for some $\beta \in [\omega_{1}, \omega_{2})$ such that, letting $\Phi^{*}$ be the set of isolated threads in $\Phi_{\beta}$,
  \begin{itemize}
  \item $\Phi^{*}$ is a proper subset of $\Phi_{\beta}$,
  \item the extension of $\langle \Phi_{\alpha} : \alpha < \beta \rangle$ by $\Phi^{*}$ gives a Scott process.
  \end{itemize}
  We could then try to build a strong weaving $\{ \phi_{a} : a \in [\omega_{1}]^{\less\omega}\}$ through $\Phi_{\beta}$, and an an uncountable set $Y \subseteq \omega_{1}$ such that $\{\phi_{a} : a \in [Y]^{\less\omega}\}$ is a strong weaving through $\Phi^{*}$ (or, more precisely, induces one via some bijection between $Y$ and $\omega_{1}$). Carrying out this construction, we come to a point where, for some infinite $\gamma < \omega_{1}$, $\phi_{a}$ has been chosen for every finite subset of $\gamma$, and for $\{\gamma\} \cup a$, for some finite $a \subseteq \gamma$ intersecting $Y$ as so far constructed, but not contained in it. At some stages it will also be that this $\gamma$ has been put into $Y$. Now suppose that $\delta$ is in $Y \cap \gamma$, as constructed so far, but that no formula for $\{\delta, \gamma\}$ has been chosen. Then we need to choose a formula for $a \cup \{\delta, \gamma\}$ such that the induced formula for $(a \cap Y) \cup \{\delta, \gamma\}$ is in $\Phi^{*}$. Since $\Phi_{\delta}$ amalgamates, we can choose a formula for $a \cup \{\delta, \gamma\}$, but we can't guarantee that the induced formula for $(a \cap Y) \cup \{\delta, \gamma\}$ will be in $\Phi^{*}$. Similarly, since $\Phi^{*}$ amalgamates we can choose a formula for $(a \cap Y)\cup \{\delta, \gamma\}$ in $\Phi^{*}$. Then we have the same 3-amalgamation issue as in Remark \ref{noomega2}, as we would then need to amalgamate the chosen formulas for $(a \cap Y)\cup \{\delta, \gamma\}$, $a \cup \{\delta\}$ and $a \cup \{\gamma\}$ in $\Phi_{\beta}$.
\end{remark}

%The connection between the results of this section and Vaught's Conjecture is taken up in Section \ref{fabsection}.

\section{Finite existential blocks}\label{qsec}

The function $E$ defined in Definition \ref{pdef} corresponds to a single existential quantifier. In this section we extend $E$ to the function $F$ which corresponds to finite blocks of existential quantifiers. The analysis of $F$ in this section is used in the following section. Most of this section consists of consequences of Proposition \ref{qlem}.

%The function $F$ here is in some sense the transitive closure of the function $E$ from Definition \ref{prodef}. While $E$ corresponds to
%a single existential quantifier, $F$ corresponds to finite sets of existential quantifiers.

\begin{df}\label{qdef} For each ordinal $\beta$, each $m \in \omega$ and each $\phi \in \Psi^{m}_{\beta}$,
  $F(\phi)$ is the set of $\psi$ such that for some $n \in \omega$ and some ordinal $\alpha$ with $\alpha + n \leq \beta$, $\psi \in \Psi^{m+n}_{\alpha}$ and there exist $\psi_{0},\ldots,\psi_{n}$ such that
  \begin{itemize}
  \item $\psi_{0} = \psi$;
  \item for all $p \in \{0,\ldots,n-1\}$, $\psi_{p} \in E(\psi_{p+1})$;
  \item $\psi_{n} = V_{\alpha+n, \beta}(\phi)$.
  \end{itemize}
\end{df}

\begin{remark}
  %Suppose that $\langle \Phi_{\alpha} : \alpha < \delta \rangle$ is a Scott process, and
  Suppose that $\beta < \delta$, $\phi \in \Phi_{\beta}$ and $\psi_{0},\ldots,\psi_{n}$ are as in Definition \ref{qdef}.
  Then by condition (\ref{econtain}) of Definition \ref{prodef},
  each $\psi_{i}$ is in $\Phi^{m+n-i}_{\alpha + i}$.
\end{remark}

\begin{remark}\label{fshift}
  Given $\alpha, \beta, \phi$ and $\psi$ as in Definition \ref{qdef}, the issue of whether or not $\psi$ is in $F(\phi)$ depends only on $V_{\alpha + n, \beta}(\phi)$ (as opposed to $\phi$). It follows that $\psi \in F(\theta)$ for any formula $\theta \in \Psi^{m}_{\gamma}$ (for some ordinal $\gamma \geq \alpha + n$) such that $V_{\alpha + n, \gamma}(\theta) = V_{\alpha + n, \beta}(\phi)$.
\end{remark}

Fix for rest of this section a Scott process $\langle \Phi_{\alpha}  :\alpha < \delta\rangle$.

\begin{prop}\label{qlem}
  %Suppose that $\langle \Phi_{\alpha}, V_{\alpha,\beta}, H_{\alpha}  :\alpha  < \beta < \delta\rangle$ is a Scott process.
   Suppose that $m,n \in \omega$ and $\alpha, \beta < \delta$ are such that $\alpha + n \leq \beta$. Let $\phi$ and $\psi$ be elements of $\Phi^{m}_{\beta}$ and $\Phi^{m+n}_{\alpha}$, respectively. Then $\psi \in F(\phi)$ if and only if there is a formula $\theta \in \Phi^{m+n}_{\beta}$ such that
  $H^{m+n}_{\beta}(\theta,i_{m}) = \phi$ and $V_{\alpha, \beta}(\theta) = \psi$.
\end{prop}

\begin{proof}
  By induction on $n$.
  %In the case $q = 0$, $\psi \in Q(\phi)$ if and only if $\psi = V_{\alpha, \beta}(\phi)$. The case then follows from the fact that %$H_{\beta}(\theta, i_{n,n}) = \theta$.
  In the case $n = 1$, $\psi \in F(\phi)$ if and only if $\psi \in E(V_{\alpha + 1, \beta}(\phi))$. In this case, the proposition is
  Proposition \ref{succcase}.
  %The reverse direction of the case $q=1$ follows from part (\ref{proppartone}) of Proposition \ref{cohprop}, with $\alpha + 1$, $\beta$ and %$i_{n,n+1}$ in the roles of $\alpha$, $\beta$ and $j$. The forward direction follows from Proposition \ref{succcase}.
  %In the case where $\beta$ is a successor ordinal, the forward direction
  %from conditions (\ref{hfour}) and (\ref{proppartthree}) of Definition \ref{prodef} (as mentioned in Proposition \ref{succcase}). In the case %where $\beta$ is a limit ordinal, the reverse direction follows from condition (\ref{ppath}) of Definition \ref{prodef}.
  The induction step from $n = p$ to $n = p + 1$ follows from the induction hypothesis in the cases $n =p$ and $n = 1$.
\end{proof}

\begin{remark}\label{flat}
  Applying Proposition \ref{qlem} and condition (\ref{htwo}) of Definition \ref{prodef}, we get that if
  $m,n,p \in \omega$ and $\alpha, \beta < \delta$ are such that $\alpha + n + p\leq \beta$, and if $\phi \in \Phi^{m}_{\beta}$, then
  for each $\psi \in \Phi^{m+n}_{\alpha} \cap F(\phi)$ there exists a $\rho \in
  \Phi^{m+n+p}_{\alpha} \cap F(\phi)$ such that $H^{m+n+p}_{\alpha}(\rho, i_{m+n}) = \psi$.
\end{remark}

\begin{remark}\label{permutef}
  Fix $m, n \in \omega$ and suppose $\alpha, \beta < \delta$ are such that $\alpha + n \leq \beta$. Let $\phi$ be an element of $\Phi^{m+n}_{\beta}$, let $f$ be an element of $\cI_{m, m+n}$ and let $g$ be any element of $\cI_{m+n,m+n}$ extending $f$. Then, by Proposition \ref{qlem} and part (\ref{hcompose}) of Remark \ref{tworems}, $V_{\alpha,\beta}(H^{m+n}_{\beta}(\phi,g))$ is in $F(H^{m+n}_{\beta}(\phi, f))$.
\end{remark}

Proposition \ref{qdownshift} follows from Proposition \ref{qlem} and Remark \ref{vdir}.

\begin{prop}\label{qdownshift}
  %Suppose that $\langle \Phi_{\alpha}, V_{\alpha,\beta}, H_{\alpha}  :\alpha  < \beta < \delta\rangle$ is a Scott process, and
  Fix $\phi \in \Phi^{m}_{\beta}$, for some $\beta < \delta$ and $m \in \omega$. Let $\psi \in \Phi^{m+n}_{\alpha}$ be an element of $F(\phi)$, for some $n \in \omega$ and some ordinal $\alpha$ with $\alpha + n \leq \beta$. Then for all $\gamma < \alpha$, $V_{\gamma, \alpha}(\psi) \in F(\phi)$.
\end{prop}

%\begin{proof}
%  When $n=1$, this is Proposition \ref{proppartthree}. The induction step from $n = p$ to $n = p + 1$ follows from the induction hypothesis in the cases $n =p$ %and the case $n = 1$.
%\end{proof}

%each pair of formulas $\phi, \psi \in \Phi_{\gamma}$, and each injection $\nu$ from the free variables of $\phi$ to the free variables of %$\psi$, $H_{\gamma}(\phi, \psi, \nu)$ holds
%if and only if, for each $\alpha < \gamma$, $H_{\alpha}(\phi', \psi', \nu)$ holds, where $\phi'$ and $\psi'$ are the members of $\Phi_{\alpha}$ %which are part of the conjunctions forming $\phi$ and $\psi$, respectively.

Propositions \ref{proppartone} and \ref{qlem} imply that members of $F(\phi)$ project horizontally to vertical projections of $\phi$.

\begin{prop} Suppose that $\alpha < \beta < \delta$, $m \leq n \in \omega$, $\phi \in \Phi^{m}_{\beta}$ and
$\psi \in \Phi^{n}_{\alpha} \cap F(\phi)$.
Then
$H^{n}_{\alpha}(\psi, i_{m}) = V_{\alpha, \beta}(\phi)$.
\end{prop}

%Together with ...., the following proposition shows how to make the (\ref{vfive}) part of the proof of Theorem \ref{limitext}
%work for an arbitrary path. (Actually, (\ref{ppath}) looks ok as well.)

Proposition \ref{qlift} is used in the proof of Theorem \ref{limitext}.

%\pagebreak

\begin{prop}\label{qlift}
  %Suppose that  $\langle \Phi_{\alpha}, V_{\alpha,\beta}, H_{\alpha}  :\alpha  < \beta < \delta\rangle$ is a Scott process,
  Suppose that $m,n \in \omega$, $\alpha < \beta$ are such that $\beta + n < \delta$, $\phi \in \Phi^{m}_{\beta + n}$ and $\psi \in \Phi^{m+n}_{\alpha} \cap F(\phi)$. Then there exists a $\psi' \in V_{\alpha, \beta}^{-1}[\{\psi\}] \cap F(\phi)$.
\end{prop}

\begin{proof}
  By Proposition \ref{qlem}, there is a $\theta \in \Phi^{m+n}_{\beta + n}$ such that $V_{\alpha, \beta + n}(\theta) = \psi$ and $H^{m+n}_{\beta + n}(\theta, i_{m}) = \phi$. By Proposition \ref{qlem} again, $V_{\beta, \beta + n}(\theta) \in F(\phi)$.
\end{proof}

A consequence of the following proposition is that every formula in a limit level of a Scott process determines the entire process below that level (note that $F(\phi)$ depends only on $\phi$). This fact is used in Remark \ref{love}.

\begin{prop}\label{fdetermine}
  Suppose that $\langle \Phi_{\alpha} : \alpha < \beta \rangle$ is a Scott process. Fix $\alpha_{0} < \beta$, $n \in \omega$ and $\phi \in \Phi^{n}_{\alpha_{0}}$. Then for each $\alpha < \beta$ and $m \in \omega$ such that $\alpha + m\leq \alpha_{0}$, the set $\Phi^{m}_{\alpha}$ is equal to $\{ H^{n}_{\alpha}(\psi, f) : \psi \in F(\phi) \cap \Phi^{n+m}_{\alpha}, f \in \cI_{m,n+m}\}$.
   %Suppose that $\alpha_{0} < \beta$ and $n \in \omega$ are such that $\alpha_{0} + n < \beta$, and let $\phi$ be an element of $\Phi^{n}_{\alpha_{0}}$
\end{prop}

\begin{proof}
  Let $\psi$ be a member of $\Phi^{m}_{\alpha}$. By condition (\ref{vfive}) of Definition \ref{prodef}, there is a $\psi' \in \Phi^{m}_{\alpha_{0}}$ such that $V_{\alpha, \alpha_{0}}(\psi') = \psi$. By condition (\ref{combine}) of Definition \ref{prodef}, there exist $f \in \cI_{m,n+m}$ and $\theta \in \Phi^{n+m}_{\alpha_{0}}$ such that $H^{n+m}_{\alpha_{0}}(\theta, i_{n}) = \phi$ and $H^{n+m}_{\alpha_{0}}(\theta, f) = \psi'$. Then $V_{\alpha, \alpha_{0}}(\theta)$ is in $F(\phi)$ by Proposition \ref{qlem}, and is as desired by Proposition \ref{proppartone}.
\end{proof}

%Proposition \ref{qext} can be proved by induction on $p$, for all $m$ and $n$ simultaneously.
%is a key part of the proof of Theorem \ref{limitext}, showing that one can include any given thread. It
%follows from Proposition \ref{qlem} and condition (\ref{combine}) of Definition \ref{prodef}.

%\pagebreak

The following proposition shows that members of $F(\phi)$ can be combined, in suitable situations.

\begin{prop}\label{qext} For all
  %\begin{itemize}
  %\item
  $m, n, p \in \omega$,
  % with $\min\{n,p\} > 0$,
  %\item $j \in \cF_{m,p}$;
  %\item $k \in \cF_{m,p}$;
  %\item
  all $\alpha, \beta < \delta$ such that $\beta \geq \alpha + n + p$, and all
  %\item
  $\phi \in \Phi^{m}_{\beta}$,
  %\item
  $\psi \in \Phi^{m+n}_{\alpha} \cap F(\phi)$ and
  %\item
  $\theta \in \Phi^{m + p}_{\alpha} \cap F(\phi)$,
  %\end{itemize}
  %such that $$V_{\alpha, \beta}(\phi) = H_{\alpha}(\psi, i_{m}) = H_{\alpha}(\theta, i_{m}),$$
  there exist
  %\begin{itemize}
  %\item
  $j \in \cF_{m+p,m+n+p}$ and
  %\item $k' \in \cF_{m + n,m+n+p}$;
  %\item
  $\rho \in \Phi^{m+n + p}_{\alpha} \cap F(\phi)$
  %\end{itemize}
  such that
  \begin{itemize}
  \item $j \circ i_{m} = i_{m}$;
  \item $H_{\alpha}(\rho, i_{m+n}) = \psi$;
  \item $H_{\alpha}(\rho, j) = \theta$.
  \end{itemize}
\end{prop}

\begin{proof}
  This can be proved by induction on $p$, for all $m$ and $n$ simultaneously. In the case where $p = 0$ there is nothing to show, so suppose that $p$ is positive. Since $\theta \in F(\phi)$, there is a $\theta' \in \Phi^{m+p-1}_{\alpha + 1}$ such that $\theta \in E(\theta')$.
  By Proposition \ref{qext}, there is a $\psi' \in \Phi^{m+n}_{\alpha + 1}$ such that $V_{\alpha, \alpha + 1}(\psi') = \psi$ and $H^{m+n}_{\alpha + 1}(\psi', i_{m}) = V_{\alpha + 1, \beta}(\phi)$. Let $\rho' \in \Phi^{m + p + p-1}_{\alpha + 1}$ be the result of applying the induction hypothesis to $\psi'$ and $\theta'$. Since $\theta \in E(\theta')$, the desired $\rho$ can be found in $E(\rho')$ by applying condition (\ref{projfact}) of Definition \ref{hdef}.
\end{proof}

The following proposition is not used in this paper. It does, however, illustrate a ways in which the function $F$ acts as expected.
The proposition follows immediately from Propositions \ref{proppartone} and \ref{qlem}.

\begin{prop}\label{qhdown}
  Suppose that $\alpha < \beta < \delta$, $m,n \in \omega$, $\phi \in \Phi^{m}_{\beta}$ and $\psi \in \Phi^{m+n}_{\alpha}$ are such that
  $\alpha + n \leq \beta$ and $\psi \in F(\phi)$. Fix $p \in [m,m+n]$ and let $j \in \cF_{p,m+n}$ be such that $j \restrict X_{m} = i_{m}$. Then $H^{m+n}_{\alpha}(\psi, j) \in F(\phi)$.
\end{prop}

\section{Extending a process of limit length}\label{ellsec}

%\begin{prop} Let $\delta$ be a limit ordinal, and let $\langle \Phi_{\alpha} : \alpha \leq \delta \rangle$ be a Scott process.  Then
%$\Phi_{\delta}$ is unified.
%\end{prop}

%\begin{proof}
%   Fix $m \leq n \in \omega$, $\phi \in \Phi^{m}_{\delta}$, $j \in \mathcal{F}_{m,n}$ and $\psi \in \Phi^{n}_{\delta}$ such that $\phi = H^{n}_{\delta}(\psi, %j)$. Fix $\theta \in \Phi^{m+1}_{\delta}$ such that $H^{m+1}_{\delta}(\theta, i_{m}) = \phi$.
%   We want to find $\rho \in \Phi^{n+1}_{\delta}$ such that $H^{n+1}_{\delta}(\rho,i_{n}) = \psi$ and $y \in X_{n+1}\setminus \range(j)$ such that
%   $H^{n+1}_{\delta}(\rho, j \cup \{(x_{m}, y)\}) = \theta$.
%
%\end{proof}

%\begin{df} For any ordinal $\alpha$, any $n \in \omega$ and any $\phi \in \Psi^{n}_{\alpha}$, we let $\Upsilon_{\alpha}(\phi)$ denote the set
%$\{H_{\alpha}(\phi, j) : m \leq n,\, j \in \cF_{m,n}\}$.
%\end{df}

\begin{df} Given a limit ordinal $\beta$ and a sequence $\langle \Phi_{\alpha} : \alpha < \beta \rangle$ such that each $\Phi_{\alpha}$ is a subset of $\Psi_{\alpha}$, a \emph{path through} $\langle \Phi_{\alpha} : \alpha < \beta \rangle$ is a formula $\phi$ in $\Psi_{\beta}$ such that $V_{\alpha, \beta}(\phi) \in \Phi_{\alpha}$ for each $\alpha < \beta$.
\end{df}

\begin{remark} For each limit ordinal $\alpha$, $\Psi_{\alpha}$ is the set of paths through the sequence $\langle \Psi_{\beta} : \beta < \alpha \rangle$.
\end{remark}

\begin{df}\label{minimalsetdef} Let $\beta$ be a limit ordinal $\beta$ and let $\langle \Phi_{\alpha} : \alpha < \beta \rangle$ be such that each $\Phi_{\alpha}$ is a subset of $\Psi_{\alpha}$. Let $\phi$ be a path through $\langle \Phi_{\alpha} : \alpha < \beta \rangle$, and let $n \in \omega$ be such that $\phi \in \Psi^{n}_{\beta}$. The \emph{minimal set} of $\phi$ according to $\langle \Phi_{\alpha} : \alpha < \beta \rangle$ is the set of paths $\upsilon$ through $\langle \Phi_{\alpha} : \alpha <\beta \rangle$ for which there exist
\begin{itemize}
\item $m \in \omega \setminus n$;
\item $p \in m + 1$;
\item $\alpha_{0} < \beta$;
\item $\psi_{0} \in \Phi^{m}_{\alpha_{0}} \cap F(\phi)$;
% such that $V_{\alpha_{0}, \beta}(\phi) = H^{m}_{\alpha_{0}}(\psi_{0}, i_{n})$;
\item $f \in \cI_{p,m}$;
\end{itemize}
such that for all $\alpha \in [\alpha_{0}, \beta)$ and all $\psi \in \Phi^{m}_{\alpha} \cap F(\phi)$ such that $V_{\alpha_{0}, \alpha}(\psi) = \psi_{0}$,
%and $H^{m}_{\alpha}(\psi, i_{n}) = V_{\alpha, \beta}(\phi)$,
$H^{m}_{\alpha}(\psi, f) = V_{\alpha, \beta}(\upsilon)$.
\end{df}

\begin{remark}
  The conclusion of the Definition \ref{minimalsetdef} can equivalently be replaced by ``such that for all $\psi \in \Psi_{\beta}$ such that $V_{\alpha_{0}, \beta}(\psi) = \psi_{0}$ and $H^{m}_{\beta}(\psi, i_{n}) = \phi$, $H^{m}_{\beta}(\psi, f) = \upsilon$."
\end{remark}

\begin{remark} Let $p \leq n$ be elements of $\omega$, let $f$ be an element of $\cI_{p,n}$, let $\beta$ be a limit ordinal, and let $\phi$ and $\phi$ be paths through a Scott process $\langle \Phi_{\alpha} : \alpha < \beta \rangle$, with $\psi$ in the minimal set of $\phi$ (with respect to this Scott process) and  $\psi \in \Psi^{n}_{\beta}$. Then $H^{n}_{\beta}(\psi, f)$ is an element of the minimal set of $\phi$.
%This can be seen by taking $m = n$ and letting $\alpha_{0}$ be any ordinal less than $\beta$.
\end{remark}

\begin{remark}\label{weakminimal setrem}
  Let the \emph{weakly minimal set} of a formula $\phi$ (in the context of Definition \ref{minimalsetdef}) be the set of formulas $\upsilon \in \Psi^{p}_{\beta}$ for which
  membership in the minimal set of $\phi$ is witnessed with $f = i_{p}$. One obtains an equivalent definition of the minimal set of $\phi$ by taking the closure of the weakly minimal set under permutations of free variables (i.e., including all formulas of the form $H^{p}_{\beta}(\upsilon, f)$, where $\upsilon \in \Psi^{p}_{\beta}$ is in the weakly minimal set of $\phi$ and $f$ is in $\cI_{p,p}$. This follows from the second part of Remark \ref{exactfree}, and condition (\ref{htwozero}) of Definition \ref{prodef}.
\end{remark}

%Let $S$ be the the set of all paths $\psi$ through $\langle \Phi_{\alpha} : \alpha < \beta \rangle$ for which there exist $\alpha < \beta$ and $f \in \cI_{m,n}$ %(where $m \leq n$ in $\omega$ are such that $\phi \in \Psi^{m}_{\beta}$ and $\psi \in \Psi^{n}_{\beta}$) such that $\psi$ is the only path $\theta$ through
%$\langle \Phi_{\alpha} : \alpha < \beta \rangle$ extending $V_{\alpha, \beta}(\psi)$ with $H^{m}_{\beta}(\theta, f) = \phi$.
%The \emph{minimal set} of $\phi$ (according to $\langle \Phi_{\alpha} : \alpha < \beta \rangle$) consists of all formulas of the form $H^{m}_{\beta}(\psi, f)$, for
%$m \in \omega \setminus n$, $\psi \in S \cap \Psi^{m}_{\beta}$ and $f \in \bigcup_{p \in (m+1)} \cI_{p,m}$.
%\end{df}

\begin{remark}\label{love} Suppose that $\beta$ is a limit ordinal, $\langle \Phi_{\alpha} : \alpha \leq \beta \rangle$ is a Scott process and $\rho$ is an element of
$\Phi_{\beta}$. Then every member of the minimal set of $\rho$ according to $\langle \Phi_{\alpha} : \alpha < \beta \rangle$ is a member of $\Phi_{\beta}$. This follows from
%conditions (\ref{vfive}) and (\ref{htwozero}) of Definition \ref{prodef}, along with Remark \ref{jandi} and
Proposition \ref{qlem}. Furthermore, by Proposition \ref{fdetermine}, $\rho$ determines $\langle \Phi_{\alpha} : \alpha < \beta \rangle$, so the expression ``according to $\langle \Phi_{\alpha} : \alpha < \beta \rangle$" is (in this case, where $\phi$ is part of a Scott process) unnecessary.
\end{remark}

We write $\ms(\phi)$ for the minimal set of $\phi$.

\begin{thrm}\label{limitext}
   Suppose that $\delta$ is a limit ordinal of countable cofinality and $\langle \Phi_{\alpha}  :\alpha  < \delta\rangle$ is a Scott process such that each $\Phi_{\alpha}$ is countable. Let $\rho$ be a path through $\langle \Phi_{\alpha} : \alpha < \delta \rangle$.  Then there exists a countable $\Phi_{\delta} \subseteq \Psi_{\delta}$ such that $\rho \in \Phi_{\delta}$ and $\langle \Phi_{\alpha}  :\alpha  \leq \delta \rangle$ is a Scott process.

 Furthermore, if $\Upsilon$ is a countable subset of $\Psi_{\delta}$ disjoint from $\ms(\rho)$, $\Phi_{\delta}$ can be chosen to be disjoint from $\Upsilon$.
\end{thrm}

\begin{proof}
In order to make $\langle \Phi_{\alpha}  :\alpha  \leq \delta \rangle$ a Scott process, we need to pick $\Phi_{\delta}$ so that conditions (\ref{vfive}), (\ref{htwozero}), (\ref{htwo}), (\ref{ppath}) and (\ref{combine}) of Definition \ref{prodef} are satisfied.
Let $\langle \gamma^{0}_{p} : p < \omega \rangle$ be an increasing sequence cofinal in $\delta$. We will recursively pick formulas $\theta_{p}$ ($p < \omega$), a nondecreasing sequence of ordinals $\gamma_{p}$ ($p < \omega$) below $\delta$ and a nondecreasing unbounded sequence of integers $n_{p}$ ($p < \omega$) such that $\rho \in \Psi^{n_{0}}_{\delta}$ and such that, for each $p \in \omega$,
\begin{itemize}
\item $\gamma_{p} \geq \gamma^{0}_{p}$;
\item $\theta_{p} \in \Phi^{n_{p}}_{\gamma_{p}} \cap F(\rho)$;
\item $H^{n_{p+1}}_{\gamma_{p}}(V_{\gamma_{p},\gamma_{p+1}}(\theta_{p+1}), i_{n_{p}}) = \theta_{p}$.
\end{itemize}
The first of these conditions implies that $\theta_{0} = V_{\gamma_{0},\delta}(\rho)$.

Having chosen the $\theta_{p}$'s, for each $n\in \omega$ we let $\phi_{n}$ be the path through $\langle \Phi_{\alpha} : \alpha < \delta \rangle$ determined by
$\{ H^{n_{p}}_{\gamma_{p}}(\theta_{p}, i_{n}) \mid p \in \omega,\, n_{p} \geq n\}$.
Then for all $m \leq n \in \omega$ we will have that $\phi_{m} = H_{\delta}(\phi_{n}, i_{m})$, and we will let $$\Phi_{\delta} = \bigcup_{n < \omega}\{H^{n}_{\delta}(\phi_{n}, j) : m \leq n,\, j \in \cF_{m,n}\}.$$

%\Upsilon_{\delta}(\phi_{n})$.

This is enough to ensure that conditions (\ref{htwozero}), (\ref{htwo}) and (\ref{combine}) from Definition \ref{prodef} are met. For condition (\ref{htwozero}) this is immediate. For condition (\ref{htwo}), the right-to-left containment follow from condition (\ref{htwozero}). For the other direction, fix $m \leq n$ in $\omega$. An arbitrary formula $\psi \in \Phi^{m}_{\delta}$ has the form $H^{q}_{\delta}(\phi_{q}, j)$, for some $q \in \omega \setminus m$ and some $j \in \cF_{m, q}$. Since $\phi_{n} = H^{p}_{\delta}(\phi_{p}, i_{n})$ for all $p \geq n$ in $\omega$, we may assume that
$q \geq n$. Letting $j' \in \cF_{n,q}$ be such that $j \restrict X_{m} = j' \restrict X_{m}$, we have that $H^{q}_{\delta}(\phi_{q}, j') \in \Phi^{n}_{\delta}$, and that $\psi = H^{n}_{\delta}(H^{q}_{\delta}(\phi_{q}, j'), i_{m})$, by part (\ref{hcompose}) of Remark \ref{tworems}.

%For the other direction, fix $m \leq n \in \omega$ and a formula $\psi \in \Phi_{\delta}$. Fix $p \in \omega$ and $k \in \cF_{m,p}$ such that $\psi = %H_{\delta}(\phi_{p}, k)$. Since $k$ is an injection, we can find a $q \in \omega$ and a $j' \in \cF_{n,q}$ such that $j' \circ i_{m} = i_{p} \circ k$. Then %$\phi_{p} = H_{\delta}(\phi_{q}, i_{p,q})$, and $\psi = H_{\delta}(\phi_{q},i_{p,q}\circ k) =
%H_{\delta}(\phi_{q},j' \circ j) = H_{\delta}(H_{\delta}(\phi_{q}, j'), j)$, by part \ref{hcompose} of Remark \ref{tworems}. Since $H_{\delta}(\phi_{q}, j') \in %\Upsilon(\phi_{q}) \subseteq \Phi_{\delta}$, we have verified this instance of condition (\ref{htwo}).

To see that condition (\ref{combine}) holds, fix $n,m \in \omega$, $\phi \in \Phi^{n}_{\delta}$ and $\psi \in \Phi^{m}_{\delta}$. Then there exist $p,q \in \omega$, $j \in \cF_{n,p}$ and $k \in \cF_{m,q}$ such that $\phi = H^{p}_{\delta}(\phi_{p}, j)$ and $\psi = H^{q}_{\delta}(\phi_{q}, k)$.
Since $$\phi_{p} = H^{\max\{p,q\}}_{\delta}(\phi_{\max\{p,q\}}, i_{p})$$ and $$\phi_{q} = H^{\max\{p,q\}}_{\delta}(\phi_{\max\{p,q\}}, i_{q}),$$ we may assume by part \ref{hcompose} of Remark \ref{tworems} that $p = q$.
%Replacing $j$ with $i_{p} \circ j$ in the case $p \leq q$, and replacing $k$ with $i_{q} \circ k$ in the case $q \leq p$, we may assume that $p = q$.
Similarly, we may assume that $p \geq m + n$. Let $A$ be a subset of $X_{p}$ of size $m + n$ which contains the ranges of both $j$ and $k$. Let $j' \colon X_{m+n} \to A$ be a bijection such that $j = j' \circ i_{n}$. Then $$\phi = H^{p}_{\delta}(\phi_{p}, j' \circ i_{n}) =
H^{m+n}_{\delta}(H^{p}_{\delta}(\phi_{p}, j'),i_{n}),$$ by part \ref{hcompose} of Remark \ref{tworems}, and $H^{p}_{\delta}(\phi_{p}, j') \in \Phi_{\delta}$. Finally, let $k' \in \cF_{m,m+n}$ be such that $k = j' \circ k'$. Then $\psi = H^{m+n}_{\delta}(H^{p}_{\delta}(\phi_{p}, j'),k')$, as desired.

To complete the proof, we show how to choose the formulas $\theta_{p}$ so that conditions (\ref{vfive}) and (\ref{ppath}) of Definition \ref{prodef} are satisfied, and also so that no member of $\Upsilon$ is in $\Phi_{\delta}$. We let $\theta_{0} = V_{\gamma_{0}, \delta}(\rho)$, as above. Suppose that $p \in \omega$ is such that $\theta_{p}$ has been chosen, but $\theta_{p+1}$ has not.

To satisfy condition (\ref{vfive}), let $\gamma_{p+1}$ be the least member of $\{ \gamma^{0}_{q} : q \in \omega\}$ which is at least as big as both $\gamma_{p}$ and $\gamma^{0}_{p + 1}$, and suppose that $\psi$ is an element of $\Phi^{m}_{\alpha}$, for some $\alpha \leq \gamma_{p + 1}$ and some $m \in \omega$.
By Proposition \ref{qlem}, we can find a formula $\theta'_{p} \in \Phi^{n_{p}}_{\gamma_{p+1} + m + n_{p}}$ such that $V_{\gamma_{p},\gamma_{p+1}+m+n_{p}}(\theta'_{p}) = \theta_{p}$ and $$H^{n_{p}}_{\gamma_{p+1} + m + n_{p}}(\theta'_{p}, i_{n_{0}}) = V_{\gamma_{p+1}+m+n_{p}, \delta}(\rho).$$ By condition (\ref{vfive}), there is a $\psi' \in \Phi^{m}_{\gamma_{p+1} + m + n_{p}}$ such that $V_{\alpha, \gamma_{p+1}+m+n_{p}}(\psi') = \psi$.
Applying condition (\ref{combine}) of Definition \ref{prodef},
%By Proposition \ref{qext},
we can choose $\theta''_{p} \in \Phi^{m+n_{p}}_{\gamma_{p + 1} + m + n_{p}}$ and $j \in \cF_{m,n_{p} + m}$ such that $$H^{m + n_{p}}_{\gamma_{p + 1} + m + n_{p}}(\theta''_{p}, i_{n_{p}}) = \theta'_{p}$$ and $$H^{m+n_{p}}_{\gamma_{p+1} + m + n_{p}}(\theta''_{p},j) = \psi'.$$ Then $\theta_{p+1} =
V_{\gamma_{p+1}, \gamma_{p + 1} + m + n_{p}}(\theta''_{p})$ is as desired, by Propositions \ref{proppartone} and \ref{qlem}.

%Let $\psi'$ be an element of $\Phi^{m}_{\gamma_{p}}$ such that $V_{\alpha, \gamma_{p}}(\psi') = \psi$. Applying condition (\ref{combine}) of Definition %\ref{prodef}, we can choose $\theta'_{p} \in \Phi^{n_{p} + m}_{\gamma_{p}}$ and $j \in \cF_{m,n_{p} + m}$ such that $H_{\gamma_{p}}(\theta'_{p}, i_{n_{p},n_{p} + %m}) = \theta_{p}$ and $H_{\gamma_{p}}(\theta'_{p},j) = \psi'$. Then we can choose $\theta_{p + 1}$ to be any element of $V_{\gamma_{p},\gamma_{p+1}}^{-1}[\{ %\theta'_{p}\}]$. Including such a step for every $\psi$ in $\bigcup_{\alpha<\delta}\Phi_{\alpha}$, we have that condition (\ref{vfive}) is satisfied.

%\pagebreak

To satisfy condition (\ref{ppath}), suppose that we have $m \leq n_{p}$
%, $j \in \cF_{m,n_{p}}$
and $\alpha < \gamma_{p}$. We can represent an
arbitrary  $\phi \in \Phi^{m}_{\delta}$ as $H^{n_{p}}_{\delta}(\phi_{p}, j)$ for some $j \in \cF_{m,n_p}$, in which case $V_{\alpha+1,\delta}(\phi)$ will
be $V_{\alpha + 1,\gamma_{p}}(H^{n_{p}}_{\gamma_{p}}(\theta_{p}, j))$. So it suffices to fix such a $j$ and a formula $\psi$ in
$E(V_{\alpha + 1, \gamma_{p}}(H^{n_{p}}_{\gamma_{p}}(\theta_{p}, j)))$. By Proposition \ref{qlift}, it suffices to find a $\theta'_{p} \in \Phi^{n_{p} + 1}_{\gamma_{p}} \cap F(\rho)$ such that $H^{n_{p} + 1}_{\gamma_{p}}(\theta'_{p}, i_{n_{p}}) = \theta_{p}$, and such that $$H^{n_{p}+1}_{\alpha}(V_{\alpha, \gamma_{p}}(\theta'_{p}), j \cup \{ (x_{m},y)\}) = \psi$$ for some $y \in X_{n_{p} + 1} \setminus \range(j)$. By Proposition \ref{proppartone}, $$V_{\alpha + 1, \gamma_{p}}(H^{n_{p}}_{\gamma_{p}}(\theta_{p}, j)) = H^{n_{p}}_{\alpha + 1}(V_{\alpha + 1, \gamma_{p}}(\theta_{p}), j).$$ By condition (\ref{projfact}) of Definition \ref{hdef}, there is a $\psi' \in E(V_{\alpha + 1, \gamma_{p}}(\theta_{p}))$ such that $$\psi =
H^{n_{p} + 1}_{\alpha}(\psi', j \cup \{(x_{m}, y)\})$$ for some $y \in X_{n_{p} + 1} \setminus \range(j)$. By Proposition \ref{qlift}, there is a $\theta^{*}_{p} \in \Phi^{n_{p}}_{\gamma_{p} + 1} \cap F(\rho)$ such that $\theta_{p} = V_{\gamma_{p}, \gamma_{p} + 1}(\theta^{*}_{p})$. By Proposition \ref{proppartthree}, there is a $\theta'_{p} \in E(\theta^{*}_{p})$ such that $V_{\alpha, \gamma_{p}}(\theta'_{p}) = \psi'$. Then $\theta'_{p}$ is as desired.

Finally let us see how to avoid the members of $\Upsilon$. Fix $m \leq n_{p}$, $f \in \cI_{m,n_{p}}$ and $\upsilon \in \Upsilon \cap \Psi^{m}_{\delta}$. It suffices to show that we can find $\gamma_{p+1}$ in the interval $(\max\{\gamma_{p}, \gamma^{0}_{p+1}\}, \delta)$ and a $\theta_{p+1} \in \Phi^{n_{p}}_{\gamma_{p + 1}} \cap F(\rho)$ such that $H^{n_{p}}_{\delta}(\theta_{p+1}, f) \neq \upsilon$. Since $\Upsilon$ is disjoint from $\ms(\rho)$, there exists such a $\theta_{p+1}$ as desired.
\end{proof}

\begin{df}\label{ecdef} Given a limit ordinal $\beta$ and sets $\Phi_{\beta}$  ($\alpha < \beta$) such that each $\Phi_{\alpha}$ is a subset of $\Psi_{\alpha}$, a path $\bigwedge\{\psi_{\alpha} : \alpha < \beta\}$ through $\langle \Phi_{\alpha} : \alpha < \beta \rangle$ is \emph{isolated} (with respect to $\langle \Phi_{\alpha} : \alpha < \beta\rangle$) if for some $\alpha_{0} < \beta$,
$|V_{\alpha_{0}, \alpha}^{-1}[\{\phi_{\alpha}\}]| = 1$ for all $\alpha \in (\alpha_{0}, \beta)$.
\end{df}

As in Remark \ref{love}, Proposition \ref{fdetermine} shows that the term ``with respect to $\langle \Phi_{\alpha} : \alpha < \beta\rangle$" is unnecessary in Definition \ref{ecdef}, if $\langle \Phi_{\alpha} : \alpha < \beta \rangle$ is a Scott process.

\begin{remark}\label{ecover}
Suppose that $\beta$ is a limit ordinal, and $\mathcal{P} = \langle \Phi_{\alpha} : \alpha < \beta \rangle$ is a Scott process. Suppose that $m \leq n$ are elements of $\omega$, $j \in \cF_{m,n}$ and $\phi \in \Psi^{n}_{\beta}$ is an isolated path through $\mathcal{P}$.
%Let $\phi$ be $\bigwedge\{\psi_{\alpha} : \alpha < \beta\}$ as in Definition \ref{ecdef}, and suppose that $n \in \omega$ is such that $\phi \in %\Psi^{n}_{\beta}$. Suppose that $m \leq n$ and $j \in \cF_{m,n}$.
Then $H^{n}_{\beta}(\phi, j)$ is isolated. To see this, note first of all that the case $m = n$ follows from part (\ref{permutevar}) of Remark \ref{tworems}. This fact allow us to reduce to the case where $j = i_{m}$. Then a proof by induction reduces to the case where $n = m + 1$. This case follows part (\ref{lf2}) of Proposition \ref{leastfixedalt}.
\end{remark}

\begin{remark}
  Given a limit ordinal $\beta$ and sets $\Phi_{\beta}$  ($\alpha < \beta$) such that each $\Phi_{\alpha}$ is a subset of $\Psi_{\alpha}$, the isolated paths through $\langle \Phi_{\alpha} : \alpha < \beta \rangle$ are exactly the minimal set of the sentence formed by taking the conjunction of the unique members of each set $\Phi^{0}_{\alpha}$. This follows from Remark \ref{ecover} and Proposition \ref{qlem}.
\end{remark}

\begin{df}
  A Scott process $\langle \Phi_{\alpha} : \alpha < \beta \rangle$ is \emph{scattered} if it there do not exist $n \in \omega$ and $\alpha_{\sigma}$, $\phi_{\sigma}$
  $(\sigma \in 2^{\less\omega})$ such that
  \begin{itemize}
  \item each $\alpha_{\sigma}$ is an element of $\beta$;
  \item each $\phi_{\sigma}$ is an element of $\Phi^{n}_{\alpha_{\sigma}}$;
  \item whenever $\sigma, \tau$ in $2^{\less\omega}$ are such that $\tau$ properly extends $\sigma$, $\alpha_{\sigma} < \alpha_{\tau}$ and $V_{\alpha_{\sigma},\alpha_{\tau}}(\phi_{\tau}) = \phi_{\sigma}$;
  \item whenever $\sigma, \tau$ in $2^{\less\omega}$ are such that neither of $\sigma$ and $\tau$ extends the other,
  \begin{itemize}
    \item if $\alpha_{\sigma} \leq \alpha_{\tau}$ then $V_{\alpha_{\sigma},\alpha_{\tau}}(\phi_{\tau}) \neq \phi_{\sigma}$,
    \item if $\alpha_{\tau} \leq \alpha_{\sigma}$ then $V_{\alpha_{\tau},\alpha_{\sigma}}(\phi_{\sigma}) \neq \phi_{\tau}$;
  \end{itemize}
  \end{itemize}
\end{df}

\begin{remark}
  Whether or not a Scott process $\langle \Phi_{\alpha} : \alpha < \beta \rangle$ is scattered is absolute between forcing extensions.
\end{remark}

\begin{remark}\label{econext} If $\langle \Phi_{\alpha}  :\alpha  < \delta\rangle$ is a Scott process of limit length having only countably many paths, it is scattered, and if it is scattered then every element of each $\Phi_{\alpha}$ is part of an isolated path through $\langle \Phi_{\alpha}  :\alpha  < \delta\rangle$. Similarly, suppose that $\langle \Phi_{\alpha}  :\alpha  < \delta\rangle$ is a Scott process of limit length where $\delta$ is possibly uncountable, and that there exist a $\beta < \omega_{1}$ such that for $\gamma \in (\beta, \omega_{1})$ there are only countably many Scott processes of length $\gamma$ extending $\langle \Phi_{\alpha} : \alpha < \beta \rangle$. Then again every element of each $\Phi_{\alpha}$ is part of an isolated path through $\langle \Phi_{\alpha}  :\alpha  < \delta\rangle$. Otherwise, one could find a $2^{\less\omega}$-splitting family in $\bigcup\{\Phi_{\alpha} : \beta < \alpha < \delta\}$ such whenever $X$ is a countable elementary submodel of a suitably large $H(\theta)$ with $X$ containing this family, the image of this splitting family under the transitive collapse of $X$ would give (using Theorem \ref{limitext}) a perfect set of Scott processes of the same countable length.
\end{remark}

%Note the that maximal completion of a Scott process of successor length $\delta$ may not be a Scott process itself. If it is, and if no initial segment of the %process is itself a maximal completion, then any model of the maximal completion has Scott rank $\delta$.

In Proposition \ref{limitecmodels}, we do not require $\delta$ to have countable cofinality (whereas we did for Theorem \ref{limitext}).

\begin{prop}\label{limitecmodels}
  Suppose that $\delta$ is a limit ordinal, and that $\langle \Phi_{\alpha}   :\alpha < \delta\rangle$ is a Scott process
  such that each element of $\bigcup\{\Phi_{\alpha} : \alpha < \delta \}$ is extended by an isolated path through $\langle \Phi_{\alpha} : \alpha < \delta \rangle$.
  %such that each $\Phi_{\alpha}$ is countable. Suppose that $\langle \Phi_{\alpha} : \alpha < \delta \rangle$ has
  %having just countably many paths.
  Letting $\Phi_{\delta}$ be the set of isolated paths through $\langle \Phi_{\alpha} : \alpha < \delta \rangle$, $\langle \Phi_{\alpha}  :\alpha  \leq  \delta\rangle$ is a Scott process. Furthermore, $\Phi_{\delta}$ then satisfies amalgamation, and every Scott process properly extending $\langle \Phi_{\alpha} : \alpha \leq \delta \rangle$ has rank at most $\delta$.
  %the extension of $\langle \Phi_{\alpha} : \alpha \leq \delta \rangle$ by the maximal completion of $\Phi_{\delta}$ is also a Scott process.
\end{prop}

\begin{proof}
  Checking that $\Phi_{\delta}$ induces a Scott process involves checking conditions (\ref{htwo}), (\ref{ppath}) and (\ref{combine}) of Definition \ref{prodef}. Remark \ref{ecover} gives one direction of (\ref{htwo}). The other conditions can be shown by applying the corresponding fact at levels above the ordinal $\alpha_{0}$ witnessing that the formulas in question are isolated.

  %Checking that the extension of $\langle \Phi_{\alpha} : \alpha \leq \delta \rangle$ by the maximal completion of $\Phi_{\delta}$ is a Scott process involves %checking that $\Phi_{\delta}$ is unified. This
  That $\Phi_{\delta}$ amalgamates also follows from the definition of the functions $H_{\alpha + 1}$ ($n \in \omega$) for any ordinal $\alpha$ above the ordinal $\alpha_{0}$ witnessing that the formulas in question are isolated. By Proposition \ref{canmax}, it also follows from the fact that some Scott properly extending $\langle \Phi_{\alpha} : \alpha \leq \delta \rangle$ has rank $\delta$, which follows from the next paragraph.

  To see that every Scott process $\langle \Phi_{\alpha} : \alpha \leq \delta + 1 \rangle$ extending $\langle \Phi_{\alpha} : \alpha \leq \delta \rangle$ has rank $\delta$, suppose that we have $n \in \omega$, $\phi \in \Phi^{n}_{\delta + 1}$ and $\psi \in \Phi^{n+1}_{\delta}$ such that $H^{n+1}_{\delta}(\psi, i_{n}) = V_{\delta, \delta + 1}(\phi)$. Let $\beta < \delta$ be such that $V_{\delta, \delta+1}(\phi)$ and $\psi$ are the unique members of $V_{\beta,\delta}^{-1}[\{V_{\beta, \delta + 1}(\phi)\}]$ and $V_{\beta, \delta}^{-1}[\{V_{\beta, \delta}(\psi)\}]$ respectively. Then $$H^{n+1}_{\beta + 1}(V_{\beta + 1, \delta}(\psi), i_{n}) = V_{\beta + 1, \delta + 1}(\phi)$$ by Proposition \ref{proppartone}, so $V_{\beta, \delta}(\psi) \in E(V_{\beta + 1, \delta+ 1}(\phi))$ by condition (\ref{hfour}) of Definition \ref{prodef}. Then conditions (\ref{hfour}) and (\ref{projfact}) of Definition \ref{prodef} imply that $\psi \in E(\phi)$.
\end{proof}

\begin{remark}\label{rhoremark}
  Theorem \ref{limitext} shows that if $\delta$ is a limit ordinal and $\langle \Phi_{\alpha} : \alpha  < \delta\rangle$ is a Scott process with just countably many paths, then for each such path $\rho$, letting $\Phi_{\delta}$ be $\ms(\rho)$ we get a Scott process $\langle \Phi_{\alpha} : \alpha \leq \delta \rangle$. Since $\ms(\phi)$ and being scattered are absolute to forcing extensions, we get the same conclusion from the assumption that $\langle \Phi_{\alpha} : \alpha  < \delta\rangle$ is scattered. In this context, then, since $\ms(\rho)$ is the smallest set one can add to $\langle \Phi_{\alpha} : \alpha < \delta \rangle$ to get a Scott processes with $\rho$ in its last level, it follows (again, in the case where $\langle \Phi_{\alpha} : \alpha  < \delta\rangle$ is scattered) that if $\phi$ and $\psi$ are paths through $\langle \Phi_{\alpha} : \alpha  < \delta\rangle$ with $\phi \in \ms(\psi)$, then $\ms(\phi)$ is a subset of $\ms(\psi)$.
\end{remark}

In the following proposition, the countability assumption on the sets $\Phi_{\alpha}$ can be replaced by the assumption that $\langle \Phi_{\alpha} : \alpha  < \gamma \rangle$ is scattered, using Remark \ref{rhoremark}.

\begin{prop}\label{laterbound} Let $\beta$ be an ordinal, and let $\gamma$ be the least limit ordinal greater than or equal to $\beta$. Suppose that $\langle \Phi_{\alpha} : \alpha \leq \gamma + 1 \rangle$ is a Scott process of pre-rank $\beta$, such that $\Phi_{\alpha}$ is countable for all $\alpha < \gamma$. Then the rank of $\langle \Phi_{\alpha} : \alpha \leq \gamma + 1 \rangle$ is at most $\gamma$.
\end{prop}

\begin{proof}
  Since $\Phi_{\alpha}$ is countable for all $\alpha < \gamma$, $\beta$ is countable.
  By the definition of pre-rank, $\langle \Phi_{\alpha} : \alpha \leq \gamma \rangle$ is the unique Scott process of length $\gamma + 1$ extending $\langle \Phi_{\alpha} : \alpha < \gamma \rangle$. By Theorem \ref{limitext} that $\langle \Phi_{\alpha} : \alpha < \gamma \rangle$ has only countably many paths. By Proposition \ref{limitecmodels} that all of them are isolated, and $\langle \Phi_{\alpha} : \alpha \leq \gamma + 1 \rangle$ has rank at most $\gamma$.
  %By Proposition \ref{limitext}, every member of $\bigcup\{\Phi_{\alpha} : \alpha  < \beta + \omega \rangle\}$ is extended by an isolated path in $\Phi_{\beta + %\omega}$. Since $\langle \Phi_{\alpha} : \alpha \leq \beta + \omega \rangle$ is the unique Scott process of length $\beta + \omega + 1$ extending $\langle %\Phi_{\alpha} : \alpha \leq \beta \rangle$, all threads in $\Phi_{\beta + \omega}$ are isolated, by Proposition \ref{limitecmodelsrho}. By Proposition %\ref{limitecmodels}, then $\langle \Phi_{\alpha} : \alpha \leq \beta + \omega + 1 \rangle$ has rank at most $\beta + \omega$.
\end{proof}

%By Proposition \ref{prerankbound}, the Scott process constructed in Proposition \ref{limitecmodelsrho} has pre-rank $\delta$ and rank at most $\delta + n$.

Combining Remark \ref{econext} with Propositions \ref{canmax}, \ref{threadcon} and \ref{limitecmodels}
%and \ref{limitecmodelsrho}
and Theorems \ref{wkh}, \ref{ctblehasmodel} and \ref{omega1model}, we get the following.

\begin{thrm}\label{alllimits} Let $\phi$ be a sentence of $\cL_{\omega_{1}, \omega}(\tau)$ and let $\alpha$ be the quantifier depth of $\phi$. Let $\beta \in (\alpha, \omega_{2})$ be an ordinal such that $\phi$ has a model of Scott rank $\beta$, but only countably many models of Scott rank $\gamma$ for each countable ordinal $\gamma$ in the interval $(\alpha, \beta)$. Then for every limit ordinal $\delta \in (\alpha, \beta)$, $\phi$ has a model of Scott rank $\delta$.
%and infinitely many models of Scott pre-rank $\delta$ which have Scott rank in the interval $[\delta, \delta + \omega)$.
\end{thrm}

%\begin{remark}\label{altamal} The material in this section gives an alternate construction of models of arbitrary cardinality of counterexamples to Vaught's %Conjecture, as it shows that the Scott process produced in Section \ref{fabsection} can be assumed to be of successor-of-limit length, with every formula at the %last level isolated. In this situation, Proposition \ref{amalgamation} can be proved from Proposition \ref{qext}.
%\end{remark}

\section{A forcing-absoluteness argument}\label{fabsection}
%\footnote{for finding Scott processes of models of analytic counterexamples to Vaught's Conjecture}

%Fix for this section a countable relational vocabulary $\tau$.

The set of $\tau$-structures with domain $\omega$ is naturally seen as a Polish space $X_{\tau}$, where a basic open set is given by the set of structures in which $R(i_{0},\ldots,i_{n-1})$ holds, for $R$ an $n$-ary relation symbol from $\tau$ and $i_{0},\ldots,i_{n-1} \in \omega$ (see Section 11.3 of \cite{Gao}, for instance). Given a sentence $\phi \in \cL_{\omega_{1}, \omega}(\tau)$, the set of models of $\phi$ (with domain $\omega$) is a Borel subset of $X_{\tau}$. By a theorem of Lopez-Escobar \cite{Lopez-Escobar}, every Borel subset of $X_{\tau}$ which is closed under isomporphism is also the set of models of some $\cL_{\omega_{1}, \omega}(\tau)$ sentence. Let us say that an \emph{analytic counterexample to Vaught's Conjecture} is an analytic set of $\tau$-structures on $\omega$, closed under isomorphism, having uncountably many models up to isomorphism, but not a perfect set of nonisomorphic models. Steel \cite{Steel} presents two examples of analytic counterexamples to Vaught's Conjecture (for certain relational vocabularies), one due to H. Friedman and the other to K. Kunen.
In this section we use a forcing-absoluteness argument to prove the following.\footnote{As mentioned in the introduction, the forcing-absoluteness arguments in this section appear in essentially identical form in Section 1 of \cite{Hjorth96}.}

\begin{thrm}\label{sec10thrm1} Suppose that $\mathcal{A}$ is an analytic counterexample to Vaught's Conjecture, and let $x \subseteq \omega$ be such that $\mathcal{A}$ is $\Sigma^{1}_{1}$ in $x$. Fix $M \in \mathcal{A}$, and let $\beta$ be an ordinal. Then $\langle \Phi_{\alpha}(M) : \alpha < \beta\rangle \in L[x]$.
\end{thrm}

Applying this theorem in forcing extensions of $V$ we get the following ostensibly stronger fact.

\begin{cor}\label{sec10cor2} Suppose that $\mathcal{A}$ is an analytic counterexample to Vaught's Conjecture, and let $x \subseteq \omega$ be such that $\mathcal{A}$ is $\Sigma^{1}_{1}$ in $x$. Let $M$ be a member of the reinterpreted version of $\mathcal{A}$ in a forcing extension of $V$, and let $\beta$ be an ordinal. Then $\langle \Phi_{\alpha}(M) : \alpha < \beta\rangle \in L[x]$.
\end{cor}

Before beginning the proof of Theorem \ref{sec10thrm1} (which is short), we make a couple remarks. In what follows we will talk of sufficient fragments of ZFC.
The theory ZFC$^{\circ}$ from \cite{BaldwinLarson} is one such fragment.

%By Theorem ...., if $\phi \in L_{\omega_{1}, \omega}(\tau)$ is a counterexample to Vaught's Conjecture, then the ranks of the Scott process

\begin{remark}\label{smallin} Let $\mathcal{A}$ be an analytic family of $\tau$-structures on $\omega$, and fix $\beta < \omega_{1}$. The set
of sequences $\langle \Phi_{\alpha}(M) : \alpha < \beta \rangle$ for $M \in \mathcal{A}$ is naturally coded by an analytic set of reals, which contains a
perfect set if it is uncountable. If $\mathcal{A}$ is an analytic counterexample to Vaught's Conjecture, then, the set of such sequences is countable for each $\beta < \omega_{1}$.

For any analytic family of $\tau$-structures, and any countable (possibly empty) set of Scott processes of length $\beta < \omega_{1}$, the assertion that there exists a member of the family whose Scott process up to length $\beta$ is not in this countable set is $\Sigma^{1}_{1}$ in codes for $\beta$, the family and the countable set, and thus absolute to any model of (a sufficient fragment of) ZFC that contains them. Furthermore, if such a model thought that uncountably many such processes existed, it could build a perfect set of such processes.
%Similarly, for each countable ordinal $\gamma$, the existence of a structure in $\mathcal{A}$ of Scott rank $\gamma$ is $\Sigma^{1}_{1}$ in $x$ and any real %coding a wellordering of $\omega$ of length $\gamma$, and thus absolute to any inner model containing such parameters.

It follows that if $\mathcal{A}$ is an analytic counterexample to Vaught's Conjecture then any inner model $N$ of (a sufficient fragment of) ZFC containing a real parameter code for $\mathcal{A}$ contains all sequences of the form $\langle \Phi_{\alpha}(M) : \alpha < \beta \rangle$, for $M \in \mathcal{A}$ and $\beta < \omega_{1}^{N}$. This gives Theorem \ref{sec10thrm1} for initial segments of Scott processes of length less than $\omega_{1}^{L[x]}$.
\end{remark}

Recall that for any ordinal $\gamma$, $\Col(\omega, \gamma)$ is the
partial order which adds a function (generically, a surjection) from
$\omega$ to $\gamma$ by finite pieces, ordered by inclusion.

\begin{proof}[Proof of Theorem \ref{sec10thrm1}.]
  Let $\theta > \beta$ be a regular cardinal of $L[x]$ such that $L_{\theta}[x]$ satisfies a sufficient fragment of ZFC (for instance, let $\theta$ be a regular cardinal of $V$ greater than $2^{2^{(|\beta| + \omega_{1})}}$).
  %Let $\langle \Phi_{\alpha} : \alpha < \beta \rangle$ be a Scott process of a structure in $\mathcal{A}$.
  Let $X$ be a countable (in $V$) elementary submodel of $L_{\theta}[x]$ containing $\{x, \langle \Phi_{\alpha} : \alpha < \beta \rangle \} \cup \beta$. Let $\gamma$ be such that the transitive collapse of $X$ is $L_{\gamma}[x]$. By the last paragraph of Remark \ref{smallin}, whenever $g$ is an $L_{\gamma}[x]$-generic filter for $\Col(\omega, \beta)$, $\langle \Phi_{\alpha} : \alpha < \beta \rangle$ is in $L_{\gamma}[x][g]$. This means that $\langle \Phi_{\alpha} : \alpha < \beta \rangle$ is in $L_{\gamma}[x]$ (this is a classical forcing fact; the point is that otherwise one could choose a generic filter while ensuring that each name in $L_{\gamma}[x]$ realizes to some value other than $\langle \Phi_{\alpha} : \alpha < \beta \rangle$). By elementarity, then,
  $\langle \Phi_{\alpha} : \alpha < \beta \rangle$ is in $L_{\theta}[x]$.
\end{proof}

\begin{remark}\label{manyranks} Let $\mathcal{A}$ be an analytic family of $\tau$-structures on $\omega$. The assertion that $\mathcal{A}$ is an analytic counterexample to Vaught's Conjecture is $\Pi^{1}_{2}$ in a real parameter $x$ for $\mathcal{A}$, and therefore absolute to $L[x]$.\footnote{There exist perfectly many nonisomorphic structures in $\mathcal{A}$ if and only if some wellfounded countable model of a sufficient fragment of ZFC thinks there exist perfectly many nonisomorphic structures in $\mathcal{A}$ (see the proof of Theorem 6.2 of \cite{BaldwinLarson}, for instance), and this later statement is easily seen to be $\Sigma^{1}_{2}$. The statement that there are countable models in $\mathcal{A}$ of unboundedly many Scott ranks below $\omega_{1}$ is easily seen to be $\Pi^{1}_{2}$.} It follows that for every cardinal $\kappa$ of $L[x]$, there are cofinally many
ordinals below $(\kappa^{+})^{L[x]}$ which are the Scott rank of a structure in $\mathcal{A}$, in any forcing extension of $L[x]$ via the partial order $\Col(\omega, \kappa)$. Applying Theorem \ref{alllimits}, this gives (in the case where $\mathcal{A}$ is Borel) that this set of ordinals (in such a forcing extension) includes coboundedly many limit ordinals below $(\kappa^{+})^{L[x]}$.
\end{remark}

Theorem \ref{settheorem} below gives an alternate proof of Theorem \ref{sec10thrm1} (the idea behind this alternate proof is essentially the same, recast slightly). The proof of Theorem \ref{settheorem} in turns uses the following standard forcing fact, which is easily proved by induction on the rank of a given set in the intersection of $V[G]$ and $V[H]$.

\begin{thrm}\label{standprod} Let $Q_{1}$ and $Q_{2}$ be partial orders, and suppose that $(G,H)$ is a generic filter for $Q_{1} \times Q_{2}$.
Then the ground model $V$ is the intersection of $V[G]$ and $V[H]$.
\end{thrm}

\begin{thrm}\label{settheorem} Suppose that $\mathcal{A}$ is an analytic counterexample to Vaught's Conjecture, and that $x \subseteq \omega$ is such that $\mathcal{A}$ is $\Sigma^{1}_{1}$ in $x$. Let $\kappa$ be an infinite cardinal of $L[x]$, fix $\gamma < (\kappa^{+})^{L[x]}$, and suppose that $\sigma$ is a
$\Col(\omega, \kappa)$-name in $L[x]$ for a $\tau$-structure in $\mathcal{A}$ of Scott rank $\gamma$. Then there is a Scott process in $L[x]$ which is forced by some
condition in $\Col(\omega, \kappa)$ to be and initial segment of the Scott process of the realization of $\sigma$.
\end{thrm}

\begin{proof}
Let $P$ be the finite support
product of $(\kappa^{+})^{L[x]}$ many copies of $\Col(\omega, \kappa)$.
%the
%partial order which adds a function (generically, a surjection) from
%$\omega$ to $\kappa$ by finite pieces, ordered by inclusion.
Forcing with $P$ over $L[x]$ makes $\kappa$ countable, but leaves $\kappa^{+}$
uncountable. Moreover, $\kappa^{+}$ is the $\omega_{1}$ of any $P$-extension.
This follows from a standard $\Delta$-system argument. Alternately, $P$ is forcing-equivalent
to the product of one copy of $\Col(\omega, \kappa)$ with the finite support product of $(\kappa^{+})^{L[x]}$ many
copies of Cohen forcing, the latter of which is c.c.c..

%By Lemma \ref{forcingpres}, there exist an ordinal $\gamma \in [\kappa, \kappa^{+})$ and a
%$\Col(\omega, \kappa)$-name $\mu$
%for a Scott process containing $\phi$ of rank $\gamma$.

Let $\mu \in L[x]$ be a $\Col(\omega, \kappa)$-name for the Scott process of the realization of $\sigma$.
The name $\mu$ induces $(\kappa^{+})^{L[x]}$ many $P$-names, by copying $\mu$ respectively
into each coordinate of $P$. Since $P$ cannot force the existence of
$(\kappa^{+})^{L[x]}$ many distinct Scott processes of rank $\gamma$ for elements of $\mathcal{A}$, there is a condition $p$
in $P$ forcing that the realizations of $\mu$ in two different coordinates will be the same.
Let $\alpha$ and $\beta$ be two such coordinates, and let $p(\alpha)$ and
$p(\beta)$ be the values of $p$ at these coordinates. Then $p(\alpha)$ and $p(\beta)$ are conditions in
$\Col(\omega, \kappa)$. By Theorem \ref{standprod}, the condition $(p(\alpha), p(\beta))$ forces in $\Col(\omega, \kappa) \times \Col(\omega, \kappa)$
that the realization of $\mu$ in each coordinate will be members of the ground model. It follows that $p(\alpha)$ and $p(\beta)$ each decide all of
$\mu$, as desired.
\end{proof}

Theorem \ref{settheorem} implies Theorem \ref{sec10thrm1}, since if Theorem \ref{sec10thrm1} were false one could let
$\sigma$ be a name for a $\tau$-structure whose Scott process (up to level $\alpha$) is not in $L[x]$, and obtain a contradiction.

Theorems \ref{omega1model} and \ref{alllimits}, along with Corollary \ref{sec10cor2} and Remark \ref{manyranks}, give the following unpublished theorem of Leo Harrington from the 1970's.

\begin{thrm}[Harrington] Suppose that $\tau$ is a countable relational vocabulary and that $\phi \in \cL_{\omega_{1}, \omega}(\tau)$ gives a counterexample to Vaught's Conjecture. Then the the Scott ranks of the models of $\phi$ of cardinality $\aleph_{1}$ are cofinal in $\omega_{2}$.
\end{thrm}

%theorem that a counterexample to Vaught's Conjecture has models of cofinally many Scott ranks below $\omega_{2}$.

The arguments we have given here give a slightly stronger version of Harrington's theorem, as follows. The theorem follows from Theorems \ref{omega1model} and \ref{alllimits}, Corollary \ref{sec10cor2}, Proposition \ref{limitecmodels} and Remark \ref{manyranks}, which together show that for all limit ordinals $\beta$ in the interval $(\alpha, \omega_{2})$, $\phi$ has a model of Scott rank $\beta$ whose formulas at level $\beta$ are all isolated. We do not know if the result stated here is new.
%Sacks \cite{Sacks} proved that a counterexample to Vaught's Conjecture has at least two models of the same Scott rank.

\begin{thrm}\label{Harringtonplus}
  Suppose that $\tau$ is a countable relational vocabulary and that $\phi \in \cL_{\omega_{1}, \omega}(\tau)$ gives a counterexample to Vaught's Conjecture. Let $\alpha$ be the quantifier depth of $\phi$. Then for every limit ordinal $\beta$ in the interval $(\alpha, \omega_{2})$, $\phi$ has a model of Scott rank $\beta$.
  %Moreover,
  %\begin{enumerate}
  %\item for club many $\beta < \omega_{2}$, $\phi$ has two nonisomorphic models of Scott rank $\beta$;
  %\item for club many $\beta < \omega_{1}$, $\phi$ has a pair $M,N$ of nonisomorphic models of Scott rank $\beta$ such that $M$ is a substructure of $N$ and $M$ %is elementary in $N$ for all formulas of quantifier depth at most $\beta$.
  %\end{enumerate}
\end{thrm}

Standard arguments show that if there is a counterexample to Vaught's Conjecture then there is one of quantifier depth at most $\omega$ (in an expanded language).

\begin{remark}\label{hmremark}
  The arguments here also give a proof of Theorem 1 of \cite{HarnikMakkai}, showing that any counterexample to Vaught's Conjecture can be strengthened to a minimal counterexample. The point again is that if $\sigma \in \cL_{\omega_{1}, \omega}(\tau)$ is a counterexample to Vaught's Conjecture, and $\alpha$ is the quantifier depth of $\sigma$, then there is a sentence $\sigma' \in \cL_{\omega_{1}, \omega}(\tau)$ which is the unique member of $\Phi^{0}_{\alpha}(M)$ for uncountably many countable models $M$ satisfying $\sigma$. Then all models of $\sigma'$ are models of $\sigma$, by Theorem \ref{wkh}, and $\sigma'$ is also a counterexample to Vaught's Conjecture. Let $S$ be the set of all countable length Scott processes which have $\sigma'$ as their unique sentence at level $\alpha$ and are initial segments of the Scott process of some model of uncountable Scott rank. Since $\sigma'$ is a counterexample to Vaught's Conjecture, $S$ is not empty. On the other hand, since $\sigma'$ does not have perfectly many countable models, there will be a member of $S$ without incompatible extensions in $S$. Since any extension of this member in $S$ will have the same property, there is such a member of $S$ with successor length. Let $\phi$ be the unique sentence in the last level of this process. Then $\phi$ is a counterexample to Vaught's Conjecture, and all uncountable models of $\phi$ satisfy the same $\cL_{\omega_{1}, \omega}(\tau)$-theory.
\end{remark}

%Furthermore, if $\kappa$ is a regular uncountable cardinal and $\langle \Phi_{\alpha} : \alpha < \kappa \rangle$ is a Scott Process with $|\Phi_{\alpha}| < %\kappa$ for all $\alpha < \kappa$, then $\Phi_{\alpha}$ amalgamates for club many $\alpha < \kappa$. This can be seen via taking elementary submodels, for %instance. If in addition $V_{\alpha, \alpha + 1}$ is not injective on $\Phi_{\alpha+1}$ for any $\alpha < \kappa$, then for no limit ordinal $\alpha < \kappa$ %will $\Phi_{\alpha}$ consist of only isolated paths. It follows then that if $\phi$ is a counterexample to Vaught's Conjecture, then for club many %$\alpha$ below each of $\omega_{1}$ and $\omega_{2}$, $\phi$ has at least two models of Scott rank $\alpha$, one ($M$) for which the Scott process at level %$\alpha$ consists only of isolated threads, and one ($N$) for which this doesn't happen.

Hjorth \cite{Hjorth} showed that if there exists a counterexample to Vaught's Conjecture, then there is one with no model of cardinality $\aleph_{2}$. Recently, this has been extended by Baldwin, S. Friedman, Koerwien and Laskowski \cite{BaldwinFriedmanKoerwienLaskowski}, who showed (among other things) that if there exists a counterexample to Vaught's Conjecture, then there is one with with the property that for some countable $\cL_{\omega_{1}, \omega}$-fragment $T$, no model of cardinality $\aleph_{1}$ has a $T$-elementary extension.

%If one could construct the models $M$ and $N$ in the previous paragraph with $M$ a substructure %of $N$, then $M$ would be elementary in $N$ for formulas of %quantifier depth $\alpha$. For countable $\alpha$, it is possible to build such a pair of models. %Doing so for an uncountable $\alpha$, however unlikely, would %prove Vaught's Conjecture.

%there is an analytic counterexample to Vaught's Conjecture without a model of cardinality $\omega_{2}$.

%\begin{remark}
%  The analysis in this section can be applied to uncountable analytic classes of countable structures containing no perfect set of nonisomorphic models. The %class of 1-transitive linear orders is such an example. Again, the class of Scott processes of all such structures in all forcing extensions exists in $L$ (there %is no need for a parameter in this case, since the class is definable without parameters). We do not know whether there is a version of Hjorth's theorem for %analytic classes.
%\end{remark}

\end{document}